\documentclass{article}

\usepackage[english]{babel}
\usepackage[a4paper,top=2cm,bottom=2cm,left=3cm,right=3cm,marginparwidth=1.75cm]{geometry}
\usepackage{amsmath,amssymb,amsthm,mathtools}
\usepackage{graphicx}
\usepackage{tikz}
\usetikzlibrary{arrows.meta}
\usepackage{indentfirst}
\usepackage{needspace}
\usepackage{setspace}
\usepackage{titlesec}
\usepackage[colorlinks=true,allcolors=blue]{hyperref}

\titlelabel{\thetitle.~}

\makeatletter
\renewenvironment{proof}[1][\proofname]{\par
  \pushQED{\qed}%
  \normalfont \topsep6\p@\@plus6\p@\relax
  \trivlist
  \item[\hskip\labelsep
        \bfseries
    #1\@addpunct{.}]\ignorespaces
}{%
  \popQED\endtrivlist\@endpefalse
}
\makeatother

\newtheorem{theorem}{Theorem}[section]
\newtheorem{prop}[theorem]{Proposition}
\newtheorem{lemma}[theorem]{Lemma}
\newtheorem{definition}[theorem]{Definition}
\newtheorem{corollary}[theorem]{Corollary}
\newtheorem{remark}[theorem]{Remark}
\newtheorem{example}[theorem]{Example}
\newtheorem{question}[theorem]{Question}
\newtheorem*{searchstrategy}{Search Strategy}

\numberwithin{equation}{section}

\DeclareMathOperator{\mDer}{\mathrm Der}
\DeclareMathOperator{\mAut}{\mathrm Aut}
\newcommand{\mG}{\mathcal{G}}
\DeclareMathOperator{\id}{\mathrm id}
\DeclareMathOperator{\mm}{\mathfrak{m}}
\DeclareMathOperator{\mF}{\mathrm F}
\DeclareMathOperator{\mJ}{\mathrm J}
\DeclareMathOperator{\rad}{Rad}

\hypersetup{pdftitle={Finite determinacy and rigidity of quiver potentials},pdfsubject={Finite determinacy of quiver potentials and rigidity}}

\makeatletter
\newcommand{\subjclass}[2][1991]{%
	\let\@oldtitle\@title%
	\gdef\@title{\@oldtitle\footnotetext{#1 \emph{Mathematics Subject Classification.} #2}}%
}
\newcommand{\keywords}[1]{%
	\let\@@oldtitle\@title%
	\gdef\@title{\@@oldtitle\footnotetext{\emph{Key words and phrases.} #1.}}%
}
\makeatother

\title{Finite determinacy and rigidity of quiver potentials}

\author{Zheng Hua, Guisong Zhou}
\date{}

\newcommand{\Addresses}{{%
		\bigskip
		\footnotesize
		
		Zheng Hua, \textsc{Department of Mathematics, The University of Hong Kong, Hong Kong.}\par
		\textit{E-mail address}: 
        \texttt{zheng.hua.zju@gmail.com}
		
		\medskip
		
		Guisong Zhou, \textsc{School of Mathematics and Statistics, Ningbo University, Ningbo, China}\par
		\textit{E-mail address}: \texttt{zhouguisong@nbu.edu.cn}
}}


\begin{document}
\maketitle
\begin{abstract}
We prove that a quiver potential over a field of characteristic zero
is finitely determined if and only if the trace space of its Jacobian
algebra is finite dimensional. In particular, Jacobi-finite potentials
are finitely determined. We further prove that rigid potentials are
finitely determined and, over an algebraically closed field, unique
up to right equivalence on a fixed quiver. For finitely determined  weighted-homogeneous potentials, we give an explicit bound on the
jet order sufficient for finite determinacy in terms of the arrow
weights, the weight of the potential, and the highest nonzero weight
of the trace space of its Jacobian algebra. We also establish new invariance results for weighted-homogeneous potentials under right equivalence.
\end{abstract}

\section{Introduction}
This paper continues our study of quiver potentials in
\cite{HZ21,HZ23}, focusing on finite determinacy and rigidity over any fields of characteristic zero.
A quiver potential is a formal linear combination of cyclic paths,
and right equivalence allows changes of arrow coordinates in the
complete path algebra. Such a potential may have infinitely many
nonzero terms, while its finite jets retain only the terms of bounded
path length. A natural question is therefore whether a sufficiently
high jet already determines the potential up to right equivalence.
This is the problem of \emph{finite determinacy}.

The problem is familiar in commutative singularity theory. For quiver
potentials, however, right equivalence acts on cyclic classes, since
potentials are considered modulo commutators. Our main result shows
that finite determinacy is characterized by the finite dimensionality
of the trace space of the Jacobian algebra, rather than by that of
the Jacobian algebra itself. We also study rigid potentials and
explicit jet orders sufficient for finite determinacy of weighted-homogeneous potentials.

Throughout, $\Bbbk$ is a field of characteristic zero and $Q$
is a finite quiver. For an algebra $A$, write
$\operatorname{Tr}(A)=A/[A,A]$ for its trace space and
$\pi_A:A\to\operatorname{Tr}(A)$ for the canonical projection,
where $[A,A]$ denotes the linear span of commutators.
Let $\mF=\Bbbk\langle\langle Q\rangle\rangle$ be the complete
path algebra, with arrow ideal $\mm$ and vertex subalgebra
$\ell=\Bbbk Q_0$. We abbreviate $\pi_{\mF}$ to $\pi$.
With its $\mm$-adic topology, $\mF$ is a complete topological
algebra, and $[\mF,\mF]$ is closed by
Corollary~\ref{cor:closed-commutators}.
We consider potentials $\Phi\in\pi(\mm^2)$. The
\emph{Jacobian algebra} of $\Phi$ is the quotient algebra
$\Lambda(\Phi)=\mF/\mJ(\Phi)$, equipped with the quotient
topology. Here, $\mJ(\Phi)$ is the \emph{Jacobian ideal} of $\Phi$,
which is the closure of the two-sided ideal generated by the cyclic
derivatives of $\Phi$. We call $\Phi$ \emph{Jacobi-finite} (resp.\ \emph{trace-finite}) if $\Lambda(\Phi)$ (resp.\ $\operatorname{Tr}(\Lambda(\Phi))$) is finite dimensional over $\Bbbk$. Two potentials
are \emph{right equivalent} if one is carried to the other by
an $\ell$-unital automorphism $\alpha$ of $\mF$, acting on the
trace space by $\pi(f)\mapsto\pi(\alpha(f))$. A potential
$\Phi$ is \emph{finitely determined} if it is $N$-determined for
some $N\geq0$, namely, if every potential $\Psi$ satisfying
$\Psi-\Phi\in\pi(\mm^{N+1})$ is right equivalent to $\Phi$.

\begin{theorem}[Theorem~\ref{thm:trace-criterion}]\label{thm:main-criterion}
A potential $\Phi\in\pi(\mm^2)$ is finitely determined if and only if
it is trace-finite.
\end{theorem}

In commutative singularity theory, a formal power series
$f\in\Bbbk[[x_1,\ldots,x_m]]$ of order at least two is finitely
right determined precisely when its Milnor algebra, the quotient
by the ideal of partial derivatives, is finite dimensional
\cite[Theorem~3.1.9]{GLS25}. The trace space of a commutative
algebra is simply the algebra itself, so
Theorem~\ref{thm:main-criterion} is a noncommutative analogue of
this classical criterion. Its distinction from the usual
Jacobi-finiteness assumption is important: a Jacobian algebra can
be much larger than its quotient by commutators. Our criterion
measures only the cyclic information needed for finite
determinacy.

The two implications in Theorem~\ref{thm:main-criterion} are proved
by different methods. For necessity, $N$-determinacy implies that, in
each finite jet space of order $n\geq N$, the affine space of
perturbations of $\Phi^{[n]}$ with the same $N$-jet lies in the
closure of its right-equivalence orbit. The tangent-space formula
for this orbit, together with the closedness of
$\pi(\mm\mJ(\Phi))$ in $\operatorname{Tr}(\mF)$, then yields
$\pi(\mm^{N+1})\subseteq\pi(\mm\mJ(\Phi))$.
By Proposition~\ref{prop:trace-criteria}, this implies that
$\Phi$ is trace-finite.
The Zariski density of $\Bbbk$-rational points in affine spaces
makes this argument valid over any field of characteristic zero,
without assuming algebraic closedness. For sufficiency, we consider
the cyclic tangent map
$\Phi_\#: \mDer_\ell(\mF)\to\operatorname{Tr}(\mF)$,
$\delta\mapsto\delta(\Phi)$, induced by the action of derivations
on cyclic classes. This is a continuous linear map between
pseudo-compact vector spaces, and its image is closed. The open
mapping theorem then allows us to eliminate higher-order terms
successively under a single inclusion
(Proposition~\ref{thm:one-inclusion}).

The criterion also has several consequences.
Since right equivalence preserves the Jacobian algebra, finite
determinacy depends only on its topological isomorphism class;
it is also preserved and reflected by arbitrary extensions of the
ground field (Corollary~\ref{cor:scalar-extension}). In
particular, Jacobi-finite potentials are finitely determined
(Corollary~\ref{cor:jacobi-fd}). This recovers over every
characteristic-zero field the result obtained by Hua and Zhou
over $\mathbb C$ using a Moser-type argument
\cite[Theorem~3.16]{HZ23}.

A particularly natural application concerns \emph{rigid}
potentials. A potential $\Phi\in\pi(\mm^2)$ is rigid if
$\pi(\mm)=\pi(\mJ(\Phi))$, or equivalently if the trace space
of $\Lambda(\Phi)$ is spanned by the vertex classes. This is the rigidity condition of Derksen, Weyman and Zelevinsky
\cite[Definition~6.10]{DWZ08}. In particular, every rigid potential is trace-finite.
We obtain the following result.

\begin{theorem}[Theorem~\ref{thm:rigid-fd}]\label{thm:main-rigid}
\begin{enumerate}
\item Every rigid potential on $Q$ is finitely determined.
\item If $\Bbbk$ is algebraically closed, any two rigid potentials on
$Q$ are right equivalent.
\end{enumerate}
\end{theorem}

Example~\ref{ex:real-signatures} shows that the algebraically
closed hypothesis cannot be dropped in general.
The uniqueness statement is especially striking because it
concerns all rigid potentials on a fixed quiver, not merely
those obtained from a particular construction. Rigid potentials
occur in substantial families: an acyclic quiver has a rigid zero
potential, and rigidity is preserved by quiver-with-potential
mutation, so the same holds for every quiver mutation equivalent
to an acyclic one \cite[Corollary~6.11]{DWZ08}.
Labardini-Fragoso also proved rigidity for the potentials
associated with ideal triangulations of marked surfaces with
nonempty boundary \cite{La09}. On the other hand, some quivers
admit no rigid potential; an example is the double cyclic
triangle of \cite[Example~8.6]{DWZ08}.
In the theory of cluster algebras, the choice of a potential is necessary for categorification. Compared to nondegenerate potentials—which always exist but are highly non-unique—rigid potentials are only expected to exist in certain cases, but they provide a canonical choice by the above theorem.

\begin{question}\label{q:jacobi-trace}
\begin{enumerate}
\item Is every rigid potential Jacobi-finite?
\item Does there exist a trace-finite potential
that is not Jacobi-finite?
\end{enumerate}
\end{question}

A negative answer to part~(1) would provide an example
for part~(2), and hence a finitely determined potential that is
not Jacobi-finite.

The qualitative finite-determinacy criterion does not by
itself specify how many terms of a potential are sufficient.
For weighted-homogeneous potentials, the grading gives additional
information that can be used to obtain an explicit bound. Assume $Q_1\ne\varnothing$. We assign positive integer
weights $\zeta=(\zeta_a)_{a\in Q_1}$ to the arrows and set
$\underline\zeta=\min_{a\in Q_1}\zeta_a$.
A potential is $\zeta$-homogeneous if its nonzero cyclic
terms have a common $\zeta$-weight. We write $d_\zeta(\Phi)$
for this weight, with the convention $d_\zeta(0)=-\infty$.
\enlargethispage{\baselineskip}% keep the decomposition sentence together
For such a potential $\Phi$, its Jacobian algebra and trace space have natural decompositions
$\Lambda(\Phi)=\prod_{d\geq0}\Lambda(\Phi)_d$ and
$\operatorname{Tr}(\Lambda(\Phi))=\prod_{d\geq0}\operatorname{Tr}(\Lambda(\Phi))_d$,
respectively, inherited from the completed $\zeta$-weight grading on $\mF$.
Their Hilbert series are then defined by
\[
H_{\zeta,\Phi}(t)=\sum_{d\geq0}\dim_\Bbbk\Lambda(\Phi)_d\,t^d,
\qquad
T_{\zeta,\Phi}(t)=\sum_{d\geq0}\dim_\Bbbk\operatorname{Tr}(\Lambda(\Phi))_d\,t^d.
\]
Neither series requires a finiteness assumption on $\Phi$. For trace-finite
$\Phi$, set
\[
 a_\zeta^{\mathrm{tr}}(\Phi)
 =\deg T_{\zeta,\Phi}(t)
 =\max\{d\geq0:\operatorname{Tr}(\Lambda(\Phi))_d\ne0\}.
\]

\begin{theorem}[Theorem~\ref{thm:weighted-general}]\label{thm:main-weighted}
Let $\Phi\in\pi(\mm^2)$ be $\zeta$-homogeneous and trace-finite.
Then $\Phi$ is $N$-determined for
\[
 N=\left\lfloor\frac{\max\{d_\zeta(\Phi),a_\zeta^{\mathrm{tr}}(\Phi)\}}
                       {\underline\zeta}\right\rfloor.
\]
\end{theorem}

The bound in Theorem~\ref{thm:main-weighted} is sharp for trace-finite
potentials homogeneous in ordinary path length: with all arrow
weights equal to one, the smallest determinacy order is
$\max\{d_\zeta(\Phi),a_\zeta^{\mathrm{tr}}(\Phi)\}$
(Corollary~\ref{cor:homogeneous-exact}). Moreover, for each fixed
$\zeta$-weight $M\geq2$, the values $a_\zeta^{\mathrm{tr}}(\Phi)$
are uniformly bounded among Jacobi-finite $\zeta$-homogeneous
potentials of weight $M$ (Proposition~\ref{prop:uniform-weight-bound}).

The graded situation also raises a question of a different
kind: which of these quantities depend only on the
right-equivalence class, when the arrow weights are kept fixed?
For weighted-homogeneous isolated singularities in the
commutative setting, both the weight and the Hilbert series
of the Milnor algebra are invariant. In fact, the weight
of the defining series and the highest nonzero weight of its Milnor
algebra determine one another explicitly (Remark~\ref{rem:commutative-weighted}).

For quiver potentials the answer is more subtle.
Theorem~\ref{thm:graded-linear-part} gives invariance when the
isomorphism induced by a right equivalence between the spaces of
indecomposables of the Jacobian algebras preserves $\zeta$-weights. This condition holds automatically when parallel
arrows have equal weights (Remark~\ref{rem:parallel-arrows-equal-weights});
Proposition~\ref{prop:graded-invariants}
shows that it also holds for one-vertex quivers when the
abelianizations of the Jacobian algebras are finite dimensional.
Thus the $\zeta$-weight and the Jacobian and trace Hilbert
series are invariant in both cases. For trace-finite potentials, so is the degree of the trace Hilbert series.
Without these restrictions, the graded quantities above need not be
right-equivalence invariants: Example~\ref{ex:two-weighted-degrees}
shows this for the $\zeta$-weight and the Jacobian Hilbert series,
whereas the counterexample in Appendix~\ref{app:trace-counterexample} shows that the
trace Hilbert series and its degree may differ even
when the $\zeta$-weight and the Jacobian Hilbert series agree.

This paper is organized as follows. Section~\ref{sec:quivers} collects the conventions and preliminary
results on complete path algebras, pseudo-compact vector spaces,
and algebraic-group orbits. Section~\ref{sec:infinitesimal}
identifies the images of cyclic tangent maps and characterizes
finite-dimensional trace spaces. Section~\ref{sec:fd} proves the
main finite-determinacy criterion. Section~\ref{sec:rigid} treats
rigidity and uniqueness. Section~\ref{sec:weighted} establishes
determinacy bounds for weighted-homogeneous potentials, and
Section~\ref{sec:graded-invariants} studies graded invariants under
right equivalence. Appendix~\ref{app:trace-counterexample} gives a
counterexample to graded trace invariance. Throughout, paths are
composed from right to left; we use $\Phi,\Psi,\Theta$ for
potentials and $\alpha,\beta$ for automorphisms.

\section{Preliminaries}\label{sec:quivers}
We begin by fixing the algebraic and topological conventions for
quiver potentials and right equivalence. The complete path algebra
and the Jacobian algebras of potentials, together with their trace
spaces, are the basic objects throughout the paper. We then recall
the pseudo-compact linear-algebra facts needed to control images of
continuous maps, and finally collect the finite-dimensional
algebraic-group results used to study orbits of truncated potentials.

\subsection{Complete path algebras and quiver potentials}
Let $Q=(Q_0,Q_1,s,t)$ be a finite quiver. Write $Q_*$ for the set of paths on $Q$,
$|u|$ for the length of a path $u$, and $e_i$ for the trivial path at
vertex $i$. Set
\[
 \mF:=\Bbbk\langle\langle Q\rangle\rangle,\qquad
 \ell=\Bbbk Q_0,\qquad \mm=(Q_1).
\]
Elements of $\mF$ are formal sums of paths. For a nonzero $f\in\mF$,
let $\operatorname{ord}(f)$ be the minimum length of a path occurring
with nonzero coefficient in $f$, and set $\operatorname{ord}(0)=+\infty$.
We compose paths from right to left: if $u:j\to i$ and $v:h\to j$, then $uv$ first follows
$v$ and then $u$; otherwise $uv=0$. In particular,
$a=e_{t(a)}ae_{s(a)}$ for every arrow $a$.

The field $\Bbbk$ is discrete, and $\mF$ has the $\mm$-adic topology.  Note that  $\rad(\mF)=\mm$, where $\rad$ denotes the Jacobson radical. Indeed, geometric-series inverses show that $\mm\subseteq\rad(\mF)$;
the semisimple quotient $\mF/\mm=\ell$ gives the reverse inclusion.
The ideal $\mm^n$ consists of formal sums of paths of length at least
$n$. Here $\mm^n$ is the algebraic ideal power: since $Q$ is
finite, every such formal sum can be grouped as
$\sum_{|p|=n}f_p p$, a finite sum over paths of length $n$. In addition,
\[
 \overline{L}=\bigcap_{n\geq0}(L+\mm^n)
\]
for every linear subspace $L\subseteq\mF$, where $\overline{L}$ denotes the closure of $L$ in $\mF$.

An \emph{$\ell$-unital algebra} is an algebra $A$
with a specified algebra homomorphism $\eta:\ell\to A$.
An \emph{$\ell$-augmented algebra} is an algebra equipped with
algebra homomorphisms $\ell\xrightarrow{\eta} A\xrightarrow{\varepsilon}\ell$
such that $\varepsilon\circ\eta=\id_\ell$. Its
\emph{augmentation ideal} is $\ker\varepsilon$.
The morphisms between $\ell$-unital (resp. $\ell$-augmented) algebras are algebra homomorphisms commuting with the specified structure maps.

The inclusion $\eta:\ell\hookrightarrow\mF$ and the
length-zero projection $\varepsilon:\mF\to\ell$ make $\mF$
an $\ell$-augmented algebra with augmentation ideal $\mm$.
Write $\mAut_\ell(\mF)$ for its $\ell$-unital algebra
automorphisms.
\begin{lemma}\label{lem:inverse}
An $\ell$-unital algebra endomorphism of $\mF$ is
automatically an $\ell$-augmented algebra endomorphism, and it
is an isomorphism if and only if its induced map on
$\mm/\mm^2$ is invertible.
\end{lemma}
\begin{proof}
Let $\alpha$ be an $\ell$-unital algebra endomorphism of $\mF$.
For each vertex $i$, let $\varepsilon_i:\mF\to\Bbbk$ take the
coefficient of $e_i$. The composite $\varepsilon_i\alpha$ is a
surjective algebra homomorphism, so its kernel is a maximal left
ideal. Hence $\mm=\rad(\mF)\subseteq\ker(\varepsilon_i\alpha)$
for every $i$. Since $\bigcap_i\ker\varepsilon_i=\mm$, we obtain
$\alpha(\mm)\subseteq\mm$. As $\alpha$ fixes $\ell$ and
$\mF=\ell\oplus\mm$, this gives
$\varepsilon\circ\alpha=\varepsilon$. Thus $\alpha$ preserves
the augmentation. Moreover,
$\alpha(\mm^n)\subseteq\mm^n$ for every $n$, so $\alpha$ is
continuous.
Clearly, if $\alpha$ is invertible then so is its induced map on
$\mm/\mm^2$. The converse implication is by
\cite[Lemma~2.12]{HZ23}.
\end{proof}

By Lemma~\ref{lem:inverse}, an $\ell$-unital algebra endomorphism
of $\mF$ is continuous and hence uniquely determined by its arrow
images. Conversely, arbitrary choices in
$e_{t(a)}\mm e_{s(a)}$ for $a\in Q_1$ extend uniquely to such an
endomorphism.
For $r\geq0$, put
\[
 \mAut_\ell^r(\mF)=
 \{\alpha\in\mAut_\ell(\mF):
       \alpha(a)-a\in\mm^{r+1}\text{ for every }a\in Q_1\}.
\]
Its elements are called unitriangular of depth at least $r$. Clearly, $\mAut_\ell^0(\mF)=\mAut_\ell(\mF)$.

\begin{lemma}[{\cite[Lemma~8.18(i)]{GLS16}}]\label{lem:products}
Let $d\geq0$ be an integer, and let $(k_r)_{r\geq0}$ be a sequence of integers
such that $k_r\geq d$ and $k_r\to\infty$. Let
$\alpha_r\in\mAut_\ell^{k_r}(\mF)$ for any $r\geq0$. Then the products
$\beta_r=\alpha_r\cdots\alpha_0$ converge uniquely to an automorphism
$\alpha\in\mAut_\ell^d(\mF)$ in the following sense: for every $N\geq0$, there exists $r_N$ such that
\[
 \beta_r(f)\equiv\alpha(f)\pmod{\mm^{N+1}}
 \qquad(f\in\mF,\ r\geq r_N).
\]
\end{lemma}

\begin{definition}\label{def:trace-space}
For an algebra $A$, its \emph{trace space} and canonical projection are
\[
 \operatorname{Tr}(A)=A/[A,A],
 \qquad \pi_A:A\longrightarrow\operatorname{Tr}(A),
\]
where $[A,A]$ denotes the linear span of commutators.
\end{definition}

For the complete path algebra $\mF$, we abbreviate $\pi_{\mF}$
to $\pi$. By Corollary~\ref{cor:closed-commutators} below,
$[\mF,\mF]$ is closed in $\mF$. The subspaces $\pi(\mm^n)$
form a neighbourhood basis of zero in $\operatorname{Tr}(\mF)$.
The cyclic classes of oriented cycles form a topological basis
of $\operatorname{Tr}(\mF)$. Consequently,
$\operatorname{Tr}(\mF)$ is a product of finite-dimensional spaces
indexed by path length. For nonzero elements $\Phi$ of $\operatorname{Tr}(\mF)$,
$\operatorname{ord}(\Phi)$ likewise denotes the least length of a nonzero component.

Each $\alpha\in\mAut_\ell(\mF)$ induces a linear
automorphism of $\operatorname{Tr}(\mF)$, giving a natural
\emph{left} action
\[
 \alpha\cdot\pi(f)=\pi(\alpha(f))
 \qquad(f\in\mF).
\]
This action preserves $\pi(\mm^n)$ for every $n\geq0$.
For a potential $\Phi$, we also write $\alpha(\Phi)$ for
$\alpha\cdot\Phi$, viewing $\alpha$ as the induced automorphism
of the trace space.

\begin{definition}
\label{def:right-equivalence}
An element of $\operatorname{Tr}(\mF)$ is called a \emph{potential}.
Two potentials $\Phi,\Psi$ are \emph{right equivalent} if
$\Phi=\alpha\cdot\Psi$ for some $\alpha\in\mAut_\ell(\mF)$.
\end{definition}

For every arrow $a$, the cyclic derivative
$\partial_a:\operatorname{Tr}(\mF)\to\mF$ is the continuous linear map given by
\[
 \partial_a(\pi(u))=\sum_{u=vaw}wv,\qquad u\in Q_*.
\]
Cyclic rotation leaves this formula unchanged, so it is well defined.
For the $n$-loop quiver, this cyclic derivative was first introduced by
Rota, Sagan and Stein~\cite{RSS80}.

\begin{definition}
For a potential $\Phi$, its \emph{Jacobian ideal} is the closure of
the two-sided ideal generated by the cyclic derivatives
$\partial_a(\Phi)$, $a\in Q_1$. Thus
\[
 \mJ(\Phi)=\overline{(\partial_a(\Phi)\mid a\in Q_1)}.
\]
The \emph{Jacobian algebra} of $\Phi$ is the quotient algebra
$\Lambda(\Phi)=\mF/\mJ(\Phi)$, equipped with the quotient topology.
The potential is \emph{Jacobi-finite} if $\Lambda(\Phi)$ is
finite dimensional over $\Bbbk$. It is \emph{trace-finite} if the trace space
$\operatorname{Tr}(\Lambda(\Phi))$ is finite dimensional over $\Bbbk$.
\end{definition}

\begin{lemma}[{\cite[Proposition~3.7]{DWZ08}}]\label{lem:jac-invariance}
For every potential $\Phi$ and every $\alpha\in\mAut_\ell(\mF)$,
\[
 \alpha(\mJ(\Phi))=\mJ(\alpha(\Phi)).
\]
Consequently, $\alpha$ induces an isomorphism
$\Lambda(\Phi)\cong\Lambda(\alpha(\Phi))$ as
$\ell$-unital algebras.
\end{lemma}

If $\Phi\in\pi(\mm^2)$, then $\mJ(\Phi)\subseteq\mm$.
Hence $\Lambda(\Phi)$ inherits an $\ell$-augmentation with
augmentation ideal $\mm/\mJ(\Phi)$. In this case, the
isomorphism induced by $\alpha$ also preserves the augmentation,
by Lemma~\ref{lem:inverse}. It is also a homeomorphism, since
$\alpha$ and its inverse are continuous.

\subsection{Pseudo-compact vector spaces}\label{sec:prelim}
Throughout this paper, topological vector spaces have linear
topologies over the discrete field $\Bbbk$: the neighbourhoods
of zero have a basis consisting of linear subspaces.
A topological vector space is called \emph{pseudo-compact}
(PC) if it is Hausdorff and complete and has a neighbourhood basis
of zero consisting of open subspaces of finite codimension
\cite[Chapter~IV, \S~3]{Ga62}; see also \cite{Va15}. It is called
\emph{linearly compact} if every family of closed affine subspaces
with the finite intersection property has nonempty intersection
\cite[\S~1]{Ze53}.

\begin{lemma}\label{lem:fip}
A Hausdorff topological vector space is
PC if and only if it is linearly compact.
\end{lemma}
\begin{proof}
If $V$ is PC, then $V\cong\varprojlim_U V/U$, where $U$ ranges
over the open subspaces and each $V/U$ is finite dimensional and
discrete. Hence $V$ is linearly compact by
\cite[Proposition~4]{Ze53}.

Conversely, let $V$ be linearly compact. Let $U$ be an open linear
subspace of $V$, and let $p:V\to V/U$ be the canonical projection.
For any family $(A_i)_{i\in I}$ of closed affine subspaces of $V/U$
with the finite intersection property, the family
$(p^{-1}(A_i))_{i\in I}$ consists of closed affine subspaces of $V$
and has the finite intersection property. By the linear compactness
of $V$, we have $\bigcap_{i\in I}p^{-1}(A_i)\ne\varnothing$,
so $\bigcap_{i\in I}A_i\ne\varnothing$. Hence $V/U$ is
linearly compact. If $V/U$ were
infinite dimensional, choose linearly independent vectors
$v_1,v_2,\ldots$ in it. The affine subspaces
\[
 (v_1+\cdots+v_n)+\operatorname{span}_{\Bbbk}
                  \{v_{n+1},v_{n+2},\ldots\},\qquad n\geq1,
\]
would have the finite intersection property but empty total
intersection. Thus every open $U$ has finite codimension.
Since the topology is linear, $V$ has a neighbourhood basis of
zero consisting of open subspaces of finite codimension.
Moreover, the compatible cosets representing any element of
$\varprojlim_U V/U$, where $U$ runs over the open subspaces of $V$,
form a family of closed affine subspaces
with the finite intersection property. Their intersection is
nonempty, so the canonical map $V\to\varprojlim_U V/U$ is
surjective. Hence $V$ is complete and therefore PC.
\end{proof}

\begin{prop}\label{prop:pc-subquot}
A closed subspace of a PC vector space is
PC. So is the quotient by that subspace.
\end{prop}
\begin{proof}
The first statement is clear. For a closed subspace $L\subseteq V$,
the quotient $V/L$ is Hausdorff. Taking inverse images of closed
 affine subspaces under the canonical projection $V\to V/L$ shows
that $V/L$ is linearly compact. Hence $V/L$ is PC by
Lemma~\ref{lem:fip}.
\end{proof}

\begin{prop}\label{prop:pc-open}
A linear subspace $L$ of a PC vector space $V$ is open if and only
if it is closed and has finite codimension.
\end{prop}
\begin{proof}
An open subspace is closed because its complement is a union of open
cosets. It has finite codimension because it contains a basic open
subspace of finite codimension.
Conversely, if a subspace $L$ is closed and has finite codimension, then $V/L$
is finite dimensional and Hausdorff, hence discrete in its linear
topology. Thus $L$ is open.
\end{proof}

\begin{remark}
Finite codimension alone implies neither openness nor closedness.
For example, let $V=\Bbbk^{\mathbb N}$ with the product topology.
It is readily seen to be PC. Choose a nonzero linear functional
$\lambda:V\to\Bbbk$
vanishing on the finite-support subspace $\Bbbk^{(\mathbb N)}$.
Such a functional exists because $V/\Bbbk^{(\mathbb N)}\ne0$.
Then $\ker\lambda$ has codimension one and contains the dense
subspace $\Bbbk^{(\mathbb N)}$. It is therefore dense and proper,
so it is neither open nor closed.
\end{remark}

\begin{prop}\label{prop:pc}
Let $f:V\to W$ be a continuous linear map between PC
vector spaces.
\begin{enumerate}
\item For every closed linear subspace $L\subseteq V$, the image $f(L)$
 is closed in $W$.
\item If $f$ is surjective, then it is open.
\end{enumerate}
\end{prop}
\begin{proof}
For (1), let $L\subseteq V$ be closed and let $w\in\overline{f(L)}$.
For every open linear subspace $U\subseteq W$, we have
$f(L)\cap(w+U)\ne\varnothing$. Hence the affine subspace
\[
 L\cap f^{-1}(w+U)
\]
is nonempty and closed. These subspaces have the finite intersection
property. Lemma~\ref{lem:fip} gives $v\in L$ with $f(v)-w\in U$
for every $U$. Hausdorffness gives $w=f(v)\in f(L)$.

For (2), let $U\subseteq V$ be an open linear subspace. By (1),
$f(U)$ is closed, and the surjection $V/U\to W/f(U)$ shows that
it has finite codimension. Hence $f(U)$ is open by
Proposition~\ref{prop:pc-open}. Since every open subset of $V$ is a union of translates of open linear subspaces, $f$ is open.
\end{proof}

The algebra $\mF$ is PC, since it is complete and Hausdorff and
each $\mm^r$ has finite codimension. By
Proposition~\ref{prop:pc-subquot}, so is every quotient algebra
$A=\mF/I$ by a closed ideal. Moreover, the trace spaces of these
algebras are PC by the following result. Recall that a topological
algebra is called \emph{topologically finitely generated} if it
contains a finitely generated dense subalgebra. The images of the
vertices and arrows topologically generate every such quotient
$\mF/I$.

\begin{corollary}\label{cor:closed-commutators}
For every topologically finitely generated PC algebra $A$, the commutator space $[A,A]$ is closed in $A$.
\end{corollary}
\begin{proof}
Choose $s_1,\ldots,s_m\in A$ generating a dense subalgebra.
The map
\[
 A^m\longrightarrow A,\qquad
 (u_1,\ldots,u_m)\longmapsto\sum_{i=1}^m[s_i,u_i]
\]
is continuous and has closed image by
Proposition~\ref{prop:pc}. Clearly, the image is contained in $[A,A]$.
The identity
$[uv,w]=[u,vw]+[v,wu]$ shows that this image contains
$[p,w]$ for every polynomial $p$ in the $s_i$ and every
$w\in A$. Since the polynomials in the $s_i$ are dense in $A$ and the image
is closed, it contains $[u,w]$ for every $u\in A$.
So $[A,A]$ is the image and hence it is closed.
\end{proof}

\subsection{Algebraic groups and orbit geometry}\label{sec:alg-group-prelim}
All schemes and algebraic groups below are over $\Bbbk$.
For a scheme $X$, the Zariski tangent space at a
$\Bbbk$-rational point $x\in X$ is denoted by $T_xX$.
If an algebraic group $G$ acts on a scheme $X$ of finite type,
then the orbit through any $x\in X(\Bbbk)$ is locally closed
\cite[Proposition~1.65(b)]{Mi17}.
We denote this orbit, endowed with its reduced subscheme structure,
by $\mathcal O_{G,x}$.

For a finite-dimensional vector space $V$, write
\[
 \mathbb A(V)=\operatorname{Spec}\bigl(\operatorname{Sym}_{\Bbbk}V^\vee\bigr).
\]
Clearly, $\mathbb A(V)$ is a commutative algebraic group, and
$\mathbb A(V)(R)=V\otimes_{\Bbbk}R$ for every commutative
$\Bbbk$-algebra $R$.  A linear subspace $E\subseteq V$
defines a closed subgroup scheme $\mathbb A(E)$, and, for
$x\in V=\mathbb A(V)(\Bbbk)$, the notation
$x+\mathbb A(E)$ means its translate under the group law.
In particular,
\[
 (x+\mathbb A(E))(\Bbbk)=x+E,\qquad
 T_x\mathbb A(V)=V,\qquad
 T_x(x+\mathbb A(E))=E.
\]
\begin{lemma}\label{lem:affine-orbit}
Let $V$ be a finite-dimensional vector space, and let $G$ be an
algebraic group acting on $\mathbb A(V)$. Then, for any $x\in V$
and any linear subspace $E\subseteq V$,
\[
 (x+\mathbb A(E))(\Bbbk)\subseteq G(\Bbbk)\cdot x
 \quad\Longrightarrow\quad E\subseteq T_x\mathcal O_{G,x}.
\]
\end{lemma}
\begin{proof}
Let $Z$ be the scheme-theoretic closure of
$\mathcal O_{G,x}$ in $\mathbb A(V)$.  Since the orbit is
locally closed and contains $x$, it is open in $Z$ near $x$,
so $T_xZ=T_x\mathcal O_{G,x}$.
If $x+\mathbb A(E)\subseteq Z$, then
\[
 E=T_x(x+\mathbb A(E))\subseteq T_xZ
    =T_x\mathcal O_{G,x}.
\]
Assume $(x+\mathbb A(E))(\Bbbk)\subseteq G(\Bbbk)\cdot x$.
As $\Bbbk$ is infinite, the set $(x+\mathbb A(E))(\Bbbk)$
is schematically dense in
$x+\mathbb A(E)$. Thus $x+\mathbb A(E)\subseteq Z$
as a closed subscheme. Hence $E\subseteq T_x\mathcal O_{G,x}$.
\end{proof}

For an $\ell$-unital algebra $A$, let $\mDer_\ell(A)$ denote the
$\Bbbk$-vector space of derivations $\delta:A\to A$ satisfying
$\delta\circ\eta=0$. Since
$\delta([u,v])=[\delta(u),v]+[u,\delta(v)]$, each such derivation
induces a linear endomorphism of $\operatorname{Tr}(A)$, given by
$\delta\cdot\pi_A(u)=\pi_A(\delta(u))$.

Now suppose that $A$ is $\ell$-augmented, with augmentation ideal $I$.
Put
\[
 \mDer^{\mathrm{aug}}_\ell(A)
 =\{\delta\in\mDer_\ell(A):\varepsilon\circ\delta=0\}
 =\{\delta\in\mDer_\ell(A):\delta(I)\subseteq I\}.
\]
For any commutative algebra $R$, define
\[
 \mathcal G_A(R)=\bigl\{\alpha\in\operatorname{Aut}_{R\text{-alg}}(A_R):
 \alpha\circ\eta_R=\eta_R,\quad
 \varepsilon_R\circ\alpha=\varepsilon_R\bigr\},
\]
where $A_R=A\otimes_{\Bbbk}R$, $\ell_R=\ell\otimes_{\Bbbk}R$,
$\eta_R=\eta\otimes_{\Bbbk}R$, and
$\varepsilon_R=\varepsilon\otimes_{\Bbbk}R$.
The assignment $R\mapsto\mathcal G_A(R)$ defines a functor from
commutative algebras to groups, via scalar extension.

\begin{lemma}
\label{lem:aug-aut-group}
Let $A$ be a finite-dimensional $\ell$-augmented algebra.
Then the functor $\mathcal G_A$ is represented by a smooth affine
algebraic group over $\Bbbk$, with
\[
 T_{\id}\mathcal G_A
 \cong\mDer^{\mathrm{aug}}_\ell(A).
\]
\end{lemma}
\begin{proof}
Let $\operatorname{GL}(A)$ be the general linear algebraic group
of the underlying $\Bbbk$-vector space of $A$, so that
$\operatorname{GL}(A)(R)=\operatorname{GL}_R(A_R)$ for every commutative
$\Bbbk$-algebra $R$. The conditions
$g(ab)=g(a)g(b)$ for $a,b\in A_R$,
$g\circ\eta_R=\eta_R$, and $\varepsilon_R\circ g=\varepsilon_R$
are polynomial equations in the matrix entries of
$g\in\operatorname{GL}_R(A_R)$. They therefore define a closed
subgroup scheme of $\operatorname{GL}(A)$ representing $\mathcal G_A$.
This group scheme is smooth by Cartier's theorem in characteristic zero
\cite[Theorem~3.23]{Mi17}.
Put $R=\Bbbk[\epsilon]/(\epsilon^2)$. An element of
$T_{\id}\mathcal G_A$ is represented by
$g\in\mathcal G_A(R)$ whose reduction modulo $\epsilon$ is
$\id$. Thus $g=\id+\epsilon\delta$ for a unique
$\Bbbk$-linear map $\delta:A\to A$. The defining equations
for $\mathcal G_A$ imply $\delta\in\mDer_\ell(A)$ and
$\varepsilon\delta=0$, equivalently $\delta(I)\subseteq I$.
This proves the tangent-space formula.
\end{proof}

For finite-dimensional $A$, we also write $\mathcal G_A$ for
its representing algebraic group. For every commutative algebra $R$,
scalar extension induces a natural identification
\[
 \operatorname{Tr}(A_R)
   \cong\operatorname{Tr}(A)\otimes_{\Bbbk}R
   =\mathbb A(\operatorname{Tr}(A))(R).
\]
The natural \emph{left} action is defined on $R$-points by
\[
 \begin{aligned}
 \mathcal G_A(R)\times
 \mathbb A(\operatorname{Tr}(A))(R)
 &\longrightarrow\mathbb A(\operatorname{Tr}(A))(R),\\
 (\alpha,\pi_{A_R}(u))
 &\longmapsto\pi_{A_R}(\alpha(u)).
 \end{aligned}
\]
Naturality in $R$ gives an algebraic group action
\[
 \mathcal G_A\times_{\Bbbk}
 \mathbb A(\operatorname{Tr}(A))
 \longrightarrow\mathbb A(\operatorname{Tr}(A)).
\]
For $x\in\operatorname{Tr}(A)$, write $\mathcal O_{A,x}$ for
the orbit subscheme $\mathcal O_{\mathcal G_A,x}\subseteq
\mathbb A(\operatorname{Tr}(A))$.

\begin{lemma}\label{lem:aug-aut-orbit}
Let $A$ be a finite-dimensional $\ell$-augmented algebra.
For $x\in\operatorname{Tr}(A)$, the orbit subscheme
$\mathcal O_{A,x}$ is smooth and locally closed, and
\[
 T_x\mathcal O_{A,x}
   =\{\delta\cdot x:\delta\in
             \mDer^{\mathrm{aug}}_\ell(A)\}.
\]
\end{lemma}
\begin{proof}
By \cite[Proposition~7.4(b)]{Mi17}, the orbit morphism
$\mu_x:\mathcal G_A\to\mathcal O_{A,x}$ is faithfully flat,
and $\mathcal O_{A,x}$ is smooth.
Let $H\subseteq\mathcal G_A$ be the stabilizer subgroup scheme of $x$.
By Cartier's theorem in characteristic zero, $H$ is smooth
\cite[Theorem~3.23]{Mi17}.
For any algebraically closed field extension $\Omega/\Bbbk$ and
$y=g\cdot x\in\mathcal O_{A,x}(\Omega)$, multiplication by $g$
induces an isomorphism
\[
 H_\Omega\xrightarrow{\sim}(\mu_x)_\Omega^{-1}(y),
 \qquad h\longmapsto gh.
\]
Thus every geometric fiber of $\mu_x$ is smooth. Since $\mu_x$ is flat and of finite presentation, it is smooth.
Consequently, its differential at the identity is surjective:
\[
 T_{\id}\mu_x:T_{\id}\mathcal G_A
 \twoheadrightarrow T_x\mathcal O_{A,x}.
\]
Under the identification
$T_{\id}\mathcal G_A\cong\mDer^{\mathrm{aug}}_\ell(A)$,
this differential sends $\delta$ to $\delta\cdot x$, as required.
\end{proof}

\section{Finite-dimensional trace spaces}\label{sec:infinitesimal}
This section is devoted to establishing a criterion for trace-finiteness. A key tool is
the linear map that each potential induces from the space of
derivations of the complete path algebra to its trace space.
We apply the theory of PC vector spaces to study the images of
these maps and obtain the desired criterion.

\begin{prop}\label{prop:trace-criteria}
Let $\Phi\in\pi(\mm^2)$. The following conditions are equivalent:
\begin{enumerate}
\item $\Phi$ is trace-finite.
\item $\displaystyle
 \dim_\Bbbk\frac{\pi(\mm)}{\pi(\mJ(\Phi))}<\infty$.
\item $\displaystyle
 \dim_\Bbbk\frac{\pi(\mm^2)}{\pi(\mm\mJ(\Phi))}<\infty$.
\item There is an $N\geq1$ such that
 $\pi(\mm^{N+1})\subseteq\pi(\mm^2\mJ(\Phi))$.
\end{enumerate}
\end{prop}

To prove the proposition, we describe the cyclic image of the
Jacobian ideal in terms of $\mDer_\ell(\mF)$.

We first equip $\mDer_\ell(\mF)$ with the structure of a PC
vector space. Every derivation $\xi\in\mDer_\ell(\mF)$ is automatically
continuous. Indeed, the Leibniz rule gives
\[
 \xi(\mm^n)\subseteq
 \sum_{j=0}^{n-1}\mm^j\xi(\mm)\mm^{n-1-j}
 \subseteq\mm^{n-1}\qquad(n\geq1).
\]
For $r\geq0$, put
\[
 \mDer_\ell^r(\mF)=
 \{\xi\in\mDer_\ell(\mF):\xi(\mm)\subseteq\mm^r\}.
\]
Thus $\mDer_\ell^0(\mF)=\mDer_\ell(\mF)$, while
$\mDer_\ell^1(\mF)=\mDer_\ell^{\mathrm{aug}}(\mF)$
consists of the augmentation-preserving derivations.
We endow $\mDer_\ell(\mF)$ with the linear topology having
$\{\mDer_\ell^r(\mF)\}_{r\geq0}$ as a neighbourhood basis of
zero.
Arrow values identify these topological vector spaces as
\begin{equation}\label{eq:der-product}
 \mDer_\ell^r(\mF)\cong
 \prod_{a\in Q_1}e_{t(a)}\mm^r e_{s(a)}
 \qquad(r\geq0).
\end{equation}
Indeed, any assignment $a\mapsto f_a\in e_{t(a)}\mF e_{s(a)}$ for $a\in Q_1$ extends by
the Leibniz rule to the ordinary path algebra, with filtration
degree at least $-1$, and hence continuously to $\mF$.
Automatic continuity gives uniqueness. The condition defining
$\mDer_\ell^r(\mF)$ is equivalent to requiring every arrow value
to lie in $\mm^r$. Each $\mDer_\ell^r(\mF)$ has finite codimension in
$\mDer_\ell(\mF)$: the quotient records only arrow values
modulo $\mm^r$. The product description also shows that the
filtration topology is complete and Hausdorff. Hence
$\mDer_\ell(\mF)$ is PC, as are its open subspaces
$\mDer_\ell^r(\mF)$.

Every derivation preserves the closed commutator subspace and
hence induces a continuous linear map on $\operatorname{Tr}(\mF)$.
For $\delta\in\mDer_\ell(\mF)$ and $\Phi=\pi(f)$, we write
$\delta(\Phi)=\delta\Phi:=\pi(\delta(f))$; this is independent
of the representative $f$. A potential $\Phi$ defines a linear map
\[
 \Phi_\#:\mDer_\ell(\mF)\longrightarrow\operatorname{Tr}(\mF),
 \qquad \xi\longmapsto\xi(\Phi)
 =\pi\left(\sum_a\xi(a)\partial_a\Phi\right).
\]
The equality follows from the Leibniz rule and cyclic rotation.
We call $\Phi_\#$ the \emph{cyclic tangent map} of $\Phi$.
For $\Phi\in\pi(\mm^2)$, its restriction to
$\mDer_\ell^1(\mF)$ takes values in $\pi(\mm^2)$.

\begin{lemma}\label{lem:cycimage}
For every $\Phi\in\operatorname{Tr}(\mF)$, the linear map
$\Phi_\#:\mDer_\ell(\mF)\to\operatorname{Tr}(\mF)$ is continuous and
\[
 \Phi_\#(\mDer_\ell^r(\mF))=\pi(\mm^r\mJ(\Phi))
 \qquad(r\geq0).
\]
Moreover, these subspaces are closed in $\operatorname{Tr}(\mF)$.
\end{lemma}
\begin{proof}
For every $r\geq0$, the formula for $\Phi_\#$ gives
\[
 \Phi_\#(\mDer_\ell^r(\mF))\subseteq\pi(\mm^r\mJ(\Phi))\subseteq\pi(\mm^r).
\]
Hence $\Phi_\#$ is continuous, and Proposition~\ref{prop:pc}(1)
shows that $\Phi_\#(\mDer_\ell^r(\mF))$ is closed.
For $x\in\mm^r$ and $y\in\mF$, cyclic rotation gives
\[
 \pi(x\partial_a\Phi\,y)=\pi(yx\partial_a\Phi).
\]
Since only the component $e_{t(a)}yx e_{s(a)}\in\mm^r$
contributes on the right, the definition of $\mJ(\Phi)$ and
closedness give
$\pi(\mm^r\mJ(\Phi))\subseteq\Phi_\#(\mDer_\ell^r(\mF))$.
\end{proof}

\Needspace{12\baselineskip}
\begin{lemma}\label{lem:jacobi-trace}
For every $\Phi\in\pi(\mm^2)$, there is a natural isomorphism
\[
 \operatorname{Tr}(\Lambda(\Phi))
 \cong \ell\oplus\frac{\pi(\mm)}{\pi(\mJ(\Phi))}.
\]
\end{lemma}
\begin{proof}
Let $p_\Phi:\mF\to\Lambda(\Phi)$ be the quotient map.
By Corollary~\ref{cor:closed-commutators}, the commutator
subspaces of $\mF$ and $\Lambda(\Phi)$ are closed.
Since $p_\Phi$ is surjective, commutators in $\Lambda(\Phi)$
lift to commutators in $\mF$. Hence the induced surjection
$(p_\Phi)_{\mathrm{tr}}:\operatorname{Tr}(\mF)\to
\operatorname{Tr}(\Lambda(\Phi))$ has kernel
\[
 \ker (p_\Phi)_{\mathrm{tr}}
 =\pi\bigl(p_\Phi^{-1}([\Lambda(\Phi),\Lambda(\Phi)])\bigr)
 =\pi(\mJ(\Phi)).
\]
Since $\operatorname{Tr}(\mF)=\ell\oplus\pi(\mm)$ and
$\mJ(\Phi)\subseteq\mm$, the assertion follows.
\end{proof}

\begin{proof}[Proof of Proposition~\ref{prop:trace-criteria}]
\emph{(1)$\Leftrightarrow$(2).}
This follows from Lemma~\ref{lem:jacobi-trace}.

\emph{(2)$\Leftrightarrow$(3).}
Since $\pi(\mm\mJ(\Phi))$ is contained in both $\pi(\mm^2)$ and
$\pi(\mJ(\Phi))$, we have two short exact sequences:
\[
\begin{aligned}
0&\longrightarrow
\frac{\pi(\mm^2)}{\pi(\mm\mJ(\Phi))}
\longrightarrow
\frac{\pi(\mm)}{\pi(\mm\mJ(\Phi))}
\longrightarrow
\frac{\pi(\mm)}{\pi(\mm^2)}
\longrightarrow 0,\\
0&\longrightarrow
\frac{\pi(\mJ(\Phi))}{\pi(\mm\mJ(\Phi))}
\longrightarrow
\frac{\pi(\mm)}{\pi(\mm\mJ(\Phi))}
\longrightarrow
\frac{\pi(\mm)}{\pi(\mJ(\Phi))}
\longrightarrow 0.
\end{aligned}
\]
The quotient $\pi(\mm)/\pi(\mm^2)$ is finite dimensional because
$Q_1$ is finite. To see that
$\pi(\mJ(\Phi))/\pi(\mm\mJ(\Phi))$ is also finite dimensional, use
the coefficient descriptions
\[
 \Phi_\#(\mDer_\ell(\mF))=\pi(\mJ(\Phi)),\qquad
 \Phi_\#(\mDer_\ell^1(\mF))=\pi(\mm\mJ(\Phi))
\]
from Lemma~\ref{lem:cycimage}. The quotient of these derivation
spaces records only constant values on loop arrows and is therefore
finite dimensional. Thus (2) and (3) are each equivalent to finite dimensionality of
$\pi(\mm)/\pi(\mm\mJ(\Phi))$.

\emph{(3)$\Leftrightarrow$(4).}
By Lemma~\ref{lem:cycimage}, the subspace
$\pi(\mm^2\mJ(\Phi))$ is closed in $\pi(\mm^2)$.
Moreover,
\[
\dim_\Bbbk\frac{\pi(\mm\mJ(\Phi))}{\pi(\mm^2\mJ(\Phi))}<\infty,
\]
since this quotient is an image of
$\mDer_\ell^1(\mF)/\mDer_\ell^2(\mF)$.
Consequently, (3) holds if and only if
$\pi(\mm^2\mJ(\Phi))$ has finite codimension in $\pi(\mm^2)$,
if and only if it is open in $\pi(\mm^2)$ by
Proposition~\ref{prop:pc-open}, if and only if it contains
$\pi(\mm^{N+1})$ for some $N\geq1$.
\end{proof}

For a field extension $\mathbb K/\Bbbk$, put
$\mF_{\mathbb K}=\mathbb K\langle\langle Q\rangle\rangle$, and denote
its arrow ideal and cyclic projection by $\mm_{\mathbb K}$ and
$\pi_{\mathbb K}$, respectively. Write $\Phi_{\mathbb K}$ for
the scalar extension of a potential $\Phi$.

\begin{prop}\label{prop:trace-base-extension}
Let $\mathbb K/\Bbbk$ be a field extension and
$\Phi\in\pi(\mm^2)$.
\begin{enumerate}
\item For fixed integers $r,N\geq0$, the inclusion
\[
 \pi(\mm^{N+1})\subseteq\pi(\mm^r\mJ(\Phi))
\]
holds if and only if the corresponding inclusion holds over
$\mathbb K$.
\item $\Phi$ is trace-finite if and only if $\Phi_{\mathbb K}$ is
 trace-finite. In this case, their trace spaces have the same dimension.
\end{enumerate}
\end{prop}
\begin{proof}
For every $n\geq1$, put
$U_n=\operatorname{Tr}(\mF)/\pi(\mm^{n+1})$ and define subspaces
\[
V_n=\frac{\pi(\mm^{N+1})+\pi(\mm^{n+1})}{\pi(\mm^{n+1})},\qquad
W_n=\frac{\pi(\mm^r\mJ(\Phi))+\pi(\mm^{n+1})}{\pi(\mm^{n+1})}.
\]
Define $U_{n,\mathbb K}$, $V_{n,\mathbb K}$, and $W_{n,\mathbb K}$
analogously over $\mathbb K$.
The path basis modulo cyclic rotation identifies
$U_{n,\mathbb K}$ with $\mathbb K\otimes_\Bbbk U_n$.
Under this identification,
\[
V_{n,\mathbb K}=\mathbb K\otimes_\Bbbk V_n,
\qquad
W_{n,\mathbb K}=\mathbb K\otimes_\Bbbk W_n,
\]
since cyclic derivatives and their products commute with scalar extension
in each finite quotient. By Lemma~\ref{lem:cycimage},
$\pi(\mm^r\mJ(\Phi))$ is closed in $\operatorname{Tr}(\mF)$. Thus
\[
\pi(\mm^r\mJ(\Phi))
=\bigcap_{n\geq N}
\bigl(\pi(\mm^r\mJ(\Phi))+\pi(\mm^{n+1})\bigr).
\]
Similarly, the same equality holds over $\mathbb K$.
Thus, to prove (1), it suffices to show
\[
V_n\subseteq W_n\quad\Longleftrightarrow\quad
V_{n,\mathbb K}\subseteq W_{n,\mathbb K}
\qquad(n\geq N),
\]
which follows from the exactness and faithfulness of the functor
$\mathbb K\otimes_\Bbbk-$.

For (2), the trace space is finite dimensional if and only if
$\pi(\mJ(\Phi))$ is open in $\operatorname{Tr}(\mF)$, by
Lemma~\ref{lem:jacobi-trace} and Proposition~\ref{prop:pc-open}.
Equivalently,
\[
 \pi(\mm^{N+1})\subseteq\pi(\mJ(\Phi))
\]
for some $N\geq1$. By (1) with $r=0$, this condition is preserved
and reflected by field extension. When it holds, with $r=0$ we have
\[
\operatorname{Tr}(\Lambda(\Phi))\cong U_N/W_N,\qquad
\operatorname{Tr}(\Lambda(\Phi_{\mathbb K}))
\cong U_{N,\mathbb K}/W_{N,\mathbb K}.
\]
The preceding identifications yield
$U_{N,\mathbb K}/W_{N,\mathbb K}
\cong\mathbb K\otimes_\Bbbk(U_N/W_N)$.
Consequently,
$\operatorname{Tr}(\Lambda(\Phi_{\mathbb K}))
\cong\mathbb K\otimes_\Bbbk\operatorname{Tr}(\Lambda(\Phi))$,
and hence their dimensions are equal.
\end{proof}

\section{Finite determinacy}\label{sec:fd}
In this section, we establish the finite-determinacy criterion,
namely, that a potential is finitely determined if and only if it is
trace-finite. For necessity,
we exploit the geometry of algebraic-group orbits on finite jet spaces.
For sufficiency, we use the cyclic tangent map $\Phi_\#$ introduced in Section~\ref{sec:infinitesimal}
to eliminate higher-order perturbations through successive formal
changes of arrows. The interplay between formal elimination and
finite-dimensional algebraic geometry is a key feature of our proof.

For $n\geq0$, set $\mF^{[n]}=\mF/\mm^{n+1}$.
It inherits the $\ell$-augmentation of $\mF$, with augmentation
ideal $\mm/\mm^{n+1}$. Let
$q_n:\operatorname{Tr}(\mF)\to\operatorname{Tr}(\mF^{[n]})$
be the natural projection.
Write $\Phi^{[n]}=q_n(\Phi)$. The path-length grading identifies
$\Phi^{[n]}$ with the polynomial truncation $\Phi_{\leq n}$, and
\[
 \ker q_n=\pi(\mm^{n+1}),\qquad
 q_n(\pi(\mm^2))\cong\pi(\mm^2)/\pi(\mm^{n+1})
 \quad(n\geq1).
\]

\begin{definition}\label{def:finite-determinacy}
For $N\geq0$, a potential $\Phi$ is \emph{$N$-determined} if $\Psi^{[N]}=\Phi^{[N]}$ implies
that $\Psi$ is right equivalent to $\Phi$, for every potential $\Psi$. It is
\emph{finitely determined} if it is $N$-determined for some $N\geq0$.
\end{definition}

\Needspace{6\baselineskip}
We begin with a necessary condition for $N$-determinacy. Since an
$N$-determined potential is $N'$-determined for every $N'\geq N$,
we may assume $N\geq1$.

\begin{prop}\label{prop:fd-necessary}
Let $\Phi\in\pi(\mm^2)$ and $N\geq1$. If $\Phi$ is $N$-determined, then
\begin{equation}\label{eq:necessary-inclusion}
 \pi(\mm^{N+1})\subseteq\pi(\mm\mJ(\Phi)).
\end{equation}
\end{prop}

We need to compute the tangent space of the orbit in each finite jet space.

\begin{lemma}\label{lem:jet-orbit}
Let $n\geq1$ and $\mG_n=\mathcal G_{\mF^{[n]}}$.
For $\Phi\in\pi(\mm^2)$, let
$\mathcal O_n(\Phi^{[n]})$ denote the orbit subscheme through
$\Phi^{[n]}$ under the natural action of $\mG_n$ on
$\mathbb A(\operatorname{Tr}(\mF^{[n]}))$.
Then
\begin{equation}\label{eq:jet-orbit-tangent}
 T_{\Phi^{[n]}}\mathcal O_n(\Phi^{[n]})
 =q_n\bigl(\pi(\mm\mJ(\Phi))\bigr).
\end{equation}
\end{lemma}
\begin{proof}
We have
\[
\begin{aligned}
T_{\Phi^{[n]}}\mathcal O_n(\Phi^{[n]})
&=\{\delta\cdot\Phi^{[n]}:\delta\in\mDer_\ell^{\mathrm{aug}}(\mF^{[n]})\}\\
&=\{(\delta\Phi)^{[n]}:\delta\in\mDer_\ell^1(\mF)\}\\
&=q_n\bigl(\pi(\mm\mJ(\Phi))\bigr).
\end{aligned}
\]
The first and third equalities follow from
Lemma~\ref{lem:aug-aut-orbit} and Lemma~\ref{lem:cycimage}, respectively.
The second follows from the surjectivity of the canonical map
$\mDer_\ell^1(\mF)\to\mDer_\ell^{\mathrm{aug}}(\mF^{[n]})$.
\end{proof}

\begin{proof}[Proof of Proposition~\ref{prop:fd-necessary}]
For $n\geq N$, put $x_n=\Phi^{[n]}$ and
$E_n=q_n(\pi(\mm^{N+1}))$.
Every $\Bbbk$-rational point of the coset subscheme
$x_n+\mathbb A(E_n)\subseteq\mathbb A(\operatorname{Tr}(\mF^{[n]}))$
is the $n$-jet
of a potential with the same $N$-jet as $\Phi$.
The $N$-determinacy assumption thus gives an inclusion of rational points
\[
 (x_n+\mathbb A(E_n))(\Bbbk)
 \subseteq\mG_n(\Bbbk)\cdot x_n.
\]
Then Lemma~\ref{lem:affine-orbit} and Lemma~\ref{lem:jet-orbit} yield
\begin{equation}\label{eq:jet-tangent-inclusion}
 E_n\subseteq T_{x_n}\mathcal O_n(x_n)
   =q_n\bigl(\pi(\mm\mJ(\Phi))\bigr).
\end{equation}
Thus, for every $n\geq N$,
\[
 \pi(\mm^{N+1})\subseteq
 \pi(\mm\mJ(\Phi))+\pi(\mm^{n+1}).
\]
Since $\pi(\mm\mJ(\Phi))$ is closed in
$\operatorname{Tr}(\mF)$ by Lemma~\ref{lem:cycimage}, it follows that
\[
 \pi(\mm^{N+1})\subseteq
 \bigcap_{n\geq N}\bigl(\pi(\mm\mJ(\Phi))+\pi(\mm^{n+1})\bigr)
 =\pi(\mm\mJ(\Phi)),
\]
which proves the assertion.
\end{proof}

We next give a sufficient condition for finite determinacy.
To establish $N$-determinacy, we seek to eliminate perturbations in
$\pi(\mm^{N+1})$ by right equivalences. The following proposition
gives such a condition and controls the depth of the correcting
automorphism.

\begin{prop}\label{thm:one-inclusion}
Let $\Phi\in\pi(\mm^2)$ and $1\leq r\leq N$.
If
\begin{equation}\label{eq:one-inclusion}
 \pi(\mm^{N+1})\subseteq\pi(\mm^{r+1}\mJ(\Phi)),
\end{equation}
then for every $\Psi$ with $\Psi^{[N]}=\Phi^{[N]}$ there is an
$\alpha\in\mAut_\ell^r(\mF)$ such that $\alpha(\Psi)=\Phi$.
\end{prop}

\begin{lemma}\label{lem:threshold}
Let $\Phi\in\pi(\mm^2)$ and $r\geq0$. Suppose that
$\pi(\mm^{r+1}\mJ(\Phi))$ is open in $\pi(\mm^2)$.
There is a strictly increasing sequence of integers
$n_p\geq2$, $p\geq r+1$, satisfying
\begin{equation}\label{eq:threshold}
 \pi(\mm^{n_p})
 \subseteq \Phi_\#(\mDer_\ell^p(\mF))
 =\pi(\mm^p\mJ(\Phi)).
\end{equation}
If $\pi(\mm^{N+1})\subseteq\pi(\mm^{r+1}\mJ(\Phi))$
for $N\geq1$, the sequence may be chosen with $n_{r+1}=N+1$.
\end{lemma}
\begin{proof}
By Lemma~\ref{lem:cycimage}, $\Phi_\#$ maps
$\mDer_\ell^{r+1}(\mF)$ continuously onto
$\pi(\mm^{r+1}\mJ(\Phi))$. Both spaces are PC, so this map is
open by Proposition~\ref{prop:pc}. For $p\geq r+1$,
$\mDer_\ell^p(\mF)$ is open in $\mDer_\ell^{r+1}(\mF)$.
Its image is consequently open in
$\pi(\mm^{r+1}\mJ(\Phi))$, and hence in $\pi(\mm^2)$ by hypothesis.
This image contains a basic neighbourhood $\pi(\mm^{n_p})$.
Choose the thresholds recursively so that they strictly increase.
The assumed inclusion permits the stated initial choice.
\end{proof}

\begin{lemma}\label{lem:correction}
Let $p\geq2$, $\delta\in\mDer^p_\ell(\mF)$, and $\Theta=\delta(\Phi)\in\pi(\mm^q)$. Then
$\exp(-\delta)\in\mAut_\ell^{p-1}(\mF)$ and
\begin{equation}\label{eq:remainder}
 \exp(-\delta)(\Phi+\Theta)-\Phi
 =\sum_{s\geq1}\frac{(-1)^s s}{(s+1)!}\delta^s(\Theta)
 \in\pi(\mm^{q+p-1}).
\end{equation}
\end{lemma}
\begin{proof}
The Leibniz rule gives
\[
 \delta(\mm^d)\subseteq\mm^{d+p-1},\qquad
 \delta(\pi(\mm^d))\subseteq\pi(\mm^{d+p-1})\quad(d\geq1).
\]
Thus the exponential series converges and defines an $\ell$-unital
algebra automorphism with inverse $\exp(\delta)$, by the iterated
Leibniz formula. Since $\exp(-\delta)(a)-a\in\mm^p$ for every
$a\in Q_1$, we have $\exp(-\delta)\in\mAut_\ell^{p-1}(\mF)$.
Using $\delta(\Phi)=\Theta$, we obtain
\[
 \exp(-\delta)(\Phi)-\Phi
 =\sum_{s\geq0}\frac{(-1)^{s+1}}{(s+1)!}\delta^s(\Theta).
\]
Adding $\exp(-\delta)(\Theta)$ gives \eqref{eq:remainder}, since
$\delta^s(\Theta)\in\pi(\mm^{q+s(p-1)})$ for $s\geq1$.
\end{proof}

\begin{proof}[Proof of Proposition~\ref{thm:one-inclusion}]
By \eqref{eq:one-inclusion},
$\pi(\mm^{r+1}\mJ(\Phi))$ is open in $\pi(\mm^2)$.
By Lemma~\ref{lem:threshold}, choose strictly increasing
integers $n_p$ for $p\geq r+1$ satisfying \eqref{eq:threshold},
with $n_{r+1}=N+1$.

Let $\Psi=\Phi+\Theta^{(0)}$ with
$\Theta^{(0)}\in\pi(\mm^{n_{r+1}})$.
For each $i\geq0$ with $\Theta^{(i)}\ne0$, put
\[
 p_i=\max\{p\geq r+1:n_p\leq\operatorname{ord}(\Theta^{(i)})\}.
\]
By \eqref{eq:threshold}, there is a
$\delta_i\in\mDer_\ell^{p_i}(\mF)$ such that
$\delta_i(\Phi)=\Theta^{(i)}$. Set
$\alpha_i=\exp(-\delta_i)$ and
\[
 \Phi+\Theta^{(i+1)}=\alpha_i(\Phi+\Theta^{(i)}).
\]
Lemma~\ref{lem:correction} yields
\[
 \operatorname{ord}(\Theta^{(i+1)})\geq
 \operatorname{ord}(\Theta^{(i)})+p_i-1,
 \qquad \alpha_i\in\mAut_\ell^{p_i-1}(\mF).
\]
If $\Theta^{(i)}=0$ for some $i$, the corresponding finite
product gives the required automorphism. Otherwise,
$p_i\geq r+1\geq2$ implies that
$\operatorname{ord}(\Theta^{(i)})\to\infty$ and hence
$p_i\to\infty$. By Lemma~\ref{lem:products}, the products
$\alpha_i\cdots\alpha_0$ converge to an automorphism
$\alpha\in\mAut_\ell^r(\mF)$, and $\alpha(\Psi)=\Phi$.
\end{proof}

\Needspace{8\baselineskip}
\begin{remark}\label{rem:gls-comparison}
In \cite[Lemma~8.19]{GLS16}, a finite potential $\Phi$ whose cycles have length at most $L$ is shown to be finitely determined if, for some $m\geq L$, every cycle $\xi$ of length greater than $m$ is cyclically equivalent to a sum $\sum_{a\in Q_1}u_a\partial_a\Phi$ satisfying, for all $a\in Q_1$,
\[
\operatorname{ord}(u_a)+\operatorname{ord}(\partial_a\Phi)
\geq|\xi|.
\] This follows from Proposition~\ref{thm:one-inclusion}, since the hypothesis implies $\pi(\mm^{m+1})\subseteq\pi(\mm^{m-L+2}\mJ(\Phi))$; take $N=m$ and $r=m-L+1$.
\end{remark}

\Needspace{6\baselineskip}
\begin{theorem}\label{thm:trace-criterion}
A potential $\Phi\in\pi(\mm^2)$ is finitely determined if and only if
it is trace-finite.
\end{theorem}
\begin{proof}
Suppose first that $\Phi$ is $N$-determined for some $N\geq1$.
By Proposition~\ref{prop:fd-necessary},
$\pi(\mm^{N+1})\subseteq\pi(\mm\mJ(\Phi))$.
Thus $\pi(\mm\mJ(\Phi))$ has finite codimension in
$\pi(\mm^2)$, so Proposition~\ref{prop:trace-criteria}(3)
shows that $\Phi$ is trace-finite.

Conversely, suppose that $\Phi$ is trace-finite.
By Proposition~\ref{prop:trace-criteria}(4), there is $N\geq1$ such that
$\pi(\mm^{N+1})\subseteq\pi(\mm^2\mJ(\Phi))$.
Proposition~\ref{thm:one-inclusion}, with $r=1$, then shows that
$\Phi$ is $N$-determined.
\end{proof}

\begin{corollary}\label{cor:jacobi-fd}
Every Jacobi-finite potential $\Phi\in\pi(\mm^2)$ is finitely determined.
\end{corollary}
\begin{proof}
If $\Lambda(\Phi)$ is finite dimensional, then so is its trace space.
Theorem~\ref{thm:trace-criterion} applies.
\end{proof}

\begin{corollary}\label{cor:scalar-extension}
Let $\mathbb K/\Bbbk$ be a field extension and
$\Phi\in\pi(\mm^2)$. Then $\Phi$ is finitely determined over
$\Bbbk$ if and only if $\Phi_{\mathbb K}$ is finitely
determined over $\mathbb K$.
\end{corollary}
\begin{proof}
The equivalence follows from Theorem~\ref{thm:trace-criterion}
and Proposition~\ref{prop:trace-base-extension}(2).
\end{proof}

\section{Applications to rigid potentials}\label{sec:rigid}
This section applies the finite-determinacy criterion to rigid potentials and establishes their uniqueness up to right equivalence over an algebraically closed field. The key idea in proving uniqueness is to use rigidity to obtain open orbits in finite jet spaces and then pass back to potentials via finite determinacy. We conclude with examples over the real numbers that illustrate the role of the ground field.

\begin{definition}
A potential $\Phi\in\operatorname{Tr}(\mF)$ is \emph{rigid} if $\Phi\in\pi(\mm^2)$ and
$\pi(\mm)=\pi(\mJ(\Phi))$. Rigidity is invariant under right equivalence
by Lemma~\ref{lem:jac-invariance}.
\end{definition}

\begin{theorem}\label{thm:rigid-fd}\label{thm:unique}
\begin{enumerate}
\item Every rigid potential is finitely determined.
\item If $\Bbbk$ is algebraically closed, any two rigid potentials
on $Q$ are right equivalent.
\end{enumerate}
\end{theorem}

\begin{proof}[Proof of Theorem~\ref{thm:rigid-fd}\textup{(1)}]
If $\Phi$ is rigid, Lemma~\ref{lem:jacobi-trace} gives
$\operatorname{Tr}(\Lambda(\Phi))\cong\ell$, so $\Phi$ is trace-finite.
Theorem~\ref{thm:trace-criterion} therefore gives finite determinacy.
\end{proof}

To prove Theorem~\ref{thm:rigid-fd}(2), we first establish
a splitting lemma. The Splitting Theorem of
\cite[Theorem~4.6]{DWZ08} decomposes a potential on a loop-free
quiver into a trivial quadratic part and a reduced part, both unique
up to right equivalence. The following lemma extends the existence
of such a decomposition to arbitrary quivers, allowing quadratic terms in loops.

\begin{lemma}\label{lem:quadratic-splitting}
Let $\Phi\in\pi(\mm^2)$.
After a linear change of arrows and a unitriangular automorphism
of $\mF$, the potential $\Phi$ is right equivalent to
a potential of the following form:
\[
 \sum_{j=1}^{M}\pi(\lambda_jx_j^2)
 +\sum_{h=1}^{N}\pi(a_hb_h)+\Phi_{\mathrm{red}},
\]
where the $x_j$ are distinct loops, $\lambda_j\in\Bbbk^\times$,
and $a_1,b_1,\ldots,a_N,b_N$ are pairwise distinct nonloop arrows,
with each $a_hb_h$ a cyclic $2$-path, and $\Phi_{\mathrm{red}}\in\pi(\mm^3)$ contains no arrow
in $Y=\{x_j,a_h,b_h\}$.
\end{lemma}
\begin{proof}
Let $\Phi_2$ be the quadratic part of $\Phi$. The terms involving
loops define a quadratic form on the space of loops at each vertex.
Since $2$ is invertible in $\Bbbk$, a linear change of loop arrows
puts each such form into diagonal form, with nonzero terms
$\sum_j\pi(\lambda_jx_j^2)$, where $\lambda_j\in\Bbbk^\times$.
For each unordered pair of distinct vertices, the remaining quadratic
terms give a bilinear pairing between the spaces of arrows in opposite
directions. Independent linear changes in these spaces allow us to write
these terms as $\sum_{h=1}^N\pi(a_hb_h)$, where the arrows
$a_1,b_1,\ldots,a_N,b_N$ are pairwise distinct and each $a_hb_h$ is a
cyclic $2$-path. Thus, after a linear change of arrows,
\[
 \Phi_2=\sum_j\pi(\lambda_jx_j^2)+\sum_h\pi(a_hb_h).
\]

Assume inductively that the terms of path lengths $3,\ldots,d-1$,
where $d\ge3$, involve no arrows in $Y$. Modulo $\pi(\mm^{d+1})$,
the path-length-$d$ terms involving $Y$ can be written cyclically
as $\sum_{y\in Y}\pi(yv_y)$, where each $v_y$ is homogeneous
of path length $d-1$.
Define $\alpha_d\in\mAut_\ell^{d-2}(\mF)$ to be the unique
automorphism determined by the arrow substitutions
\[
 x_j\longmapsto x_j-\frac{1}{2\lambda_j}v_{x_j},\qquad
 a_h\longmapsto a_h-v_{b_h},\qquad
 b_h\longmapsto b_h-v_{a_h},
\]
with all arrows outside $Y$ fixed. Then
\[
 \alpha_d(\Phi)\equiv
 \Phi-\sum_{y\in Y}\pi(yv_y)
 \pmod{\pi(\mm^{d+1})}.
\]
Thus the terms of path lengths $3,\ldots,d$ in $\alpha_d(\Phi)$
involve no arrows in $Y$. Applying this construction successively
for $d=3,4,\ldots$, Lemma~\ref{lem:products} yields an
automorphism $\alpha\in\mAut_\ell^1(\mF)$ such that
$\alpha(\Phi)=\Phi_2+\Phi_{\mathrm{red}}$, where
$\Phi_{\mathrm{red}}\in\pi(\mm^3)$ involves no arrow in $Y$.
\end{proof}

\begin{lemma}\label{prop:rigid-tangent}
A potential $\Phi\in\pi(\mm^2)$ is rigid if and only if
$\pi(\mm^2)=\pi(\mm\mJ(\Phi))$.
\end{lemma}
\begin{proof}
Since automorphisms preserve $\mm$, both conditions are invariant
under right equivalence by Lemma~\ref{lem:jac-invariance}.
By Lemma~\ref{lem:quadratic-splitting}, we may therefore assume
$\Phi$ is of the following form:
\[
 \Phi=\sum_{j=1}^{M}\pi(\lambda_jx_j^2)
       +\sum_{h=1}^{N}\pi(a_hb_h)+\Phi_0,
\]
where the arrows $x_j,a_h,b_h$ and the coefficients $\lambda_j$ are as in
Lemma~\ref{lem:quadratic-splitting}, and $\Phi_0\in\pi(\mm^3)$
contains no arrow in $Y=\{x_1,\ldots,x_M,a_1,b_1,\ldots,a_N,b_N\}$.
Its Jacobian ideal is
\[
 \mJ(\Phi)=\overline{(Y)}+\mJ(\Phi_0),\qquad \mJ(\Phi_0)\subseteq\mm^2.
\]
Thus, for any loop $c\in Q_1$, if $\pi(c)\in\pi(\mJ(\Phi))$ or
$\pi(c^2)\in\pi(\mm\mJ(\Phi))$, then $c=x_j$ for some $j$.

Suppose first that $\pi(\mm^2)\subseteq\pi(\mm\mJ(\Phi))$.
The preceding observation forces every loop to be one of the $x_j$,
which lie in $\mJ(\Phi)$. Every nonloop arrow has zero cyclic class.
All cyclic classes of length at least two also lie in
$\pi(\mJ(\Phi))$. Consequently $\pi(\mm)\subseteq\pi(\mJ(\Phi))$, so $\Phi$ is rigid.

Conversely, suppose that $\Phi$ is rigid. Since
$\pi(\mm)\subseteq\pi(\mJ(\Phi))$, the preceding observation
forces every loop to be one of the $x_j$. For each $\Theta\in\pi(\mm^2)$,
one has an expression
\[
 \Theta=\pi\left(\sum_{c\in Q_1} u_c\partial_c\Phi\right),
 \qquad u_c\in e_{t(c)}\mF e_{s(c)}.
\]
For nonloops $c$, the coefficient $u_c$ belongs to $\mm$. For a loop
$x_j$, write $u_{x_j}^{(0)}$ for the constant part of $u_{x_j}$. Since
$\partial_{x_j}\Phi=2\lambda_jx_j$, the path-length-one part of this equality is
\[
 \pi\left(\sum_{j=1}^{M} u_{x_j}^{(0)}2\lambda_jx_j\right)=0.
\]
The loop classes are linearly independent, so the constants $u_{x_j}^{(0)}$ vanish.
Thus every coefficient belongs to $\mm$ and
$\Theta\in\pi(\mm\mJ(\Phi))$. Hence
$\pi(\mm^2)\subseteq\pi(\mm\mJ(\Phi))$.
\end{proof}

\begin{proof}[Proof of Theorem~\ref{thm:unique}\textup{(2)}]
Let $\Phi$ and $\Psi$ be rigid potentials. By
Theorem~\ref{thm:rigid-fd}(1), choose $N\geq1$ such that $\Phi$ is
$N$-determined. Lemma~\ref{lem:jet-orbit} and
Lemma~\ref{prop:rigid-tangent} give
\[
 T_{\Phi^{[N]}}\mathcal O_N(\Phi^{[N]})
 =q_N(\pi(\mm\mJ(\Phi)))
 =q_N(\pi(\mm^2)).
\]
Since the $\mG_N$-action preserves the square of the
augmentation ideal, the orbit $\mathcal O_N(\Phi^{[N]})$ is contained in
$\mathbb A(q_N(\pi(\mm^2)))$. It is smooth and locally
closed, and its tangent dimension equals the dimension of
this irreducible affine space. The orbit $\mathcal O_N(\Phi^{[N]})$ is therefore a
nonempty open subscheme of $\mathbb A(q_N(\pi(\mm^2)))$. The same argument
applies to $\mathcal O_N(\Psi^{[N]})$. These two open
orbit subschemes must intersect. Over the algebraically
closed field $\Bbbk$ this means that the two $\Bbbk$-rational points $\Phi^{[N]}$ and $\Psi^{[N]}$ belong
to the same $\mG_N(\Bbbk)$-orbit; here we use the distinction
between an orbit subscheme and its rational points established
in Subsection~\ref{sec:alg-group-prelim}.
Choose an element of $\mG_N(\Bbbk)$ taking
$\Psi^{[N]}$ to $\Phi^{[N]}$. Lifting its arrow images yields
an $\ell$-unital algebra endomorphism $\alpha:\mF\to\mF$,
which is an automorphism by Lemma~\ref{lem:inverse}. Then
$(\alpha(\Psi))^{[N]}=\Phi^{[N]}$. By $N$-determinacy, there is
$\beta\in\mAut_\ell(\mF)$ such that $\beta(\alpha(\Psi))=\Phi$.
Thus $\beta\alpha$ is a right equivalence from $\Psi$ to $\Phi$.
\end{proof}

\begin{remark}
Over an algebraically closed field, all rigid potentials on a fixed
quiver are right equivalent by Theorem~\ref{thm:unique}(2),
and hence have isomorphic Jacobian algebras by
Lemma~\ref{lem:jac-invariance}. Moreover, since $\lambda\Phi$ is rigid whenever $\Phi$ is rigid,
Theorem~\ref{thm:unique}(2) implies that $\Phi$ and $\lambda\Phi$
are right equivalent for every $\lambda\in\Bbbk^\times$.
Examples of rigid potentials on $2$-acyclic quivers, including
quivers not mutation equivalent to acyclic ones, appear in
\cite[Examples~8.3 and~8.7]{DWZ08}. Further examples arise
from ideal triangulations of marked surfaces with nonempty
boundary \cite{La09}.
\end{remark}

\begin{example}
\label{ex:real-signatures}\normalfont
Let $\Bbbk=\mathbb R$, and let $Q$ be the quiver with one vertex and
$n\geq1$ loops $x_1,\ldots,x_n$. For $0\leq p\leq n$, set
\[
 W_p=\pi\left(\sum_{i=1}^{p}x_i^2
                   -\sum_{i=p+1}^{n}x_i^2\right).
\]
Each $W_p$ is rigid: its cyclic derivatives are $2x_i$ for
$i\leq p$ and $-2x_i$ for $i>p$, so $\mJ(W_p)=\mm$.
The $W_p$ are pairwise not right equivalent over $\mathbb R$.
Suppose there is an automorphism $\alpha\in\mAut_\ell(\mF)$ such that
$\alpha W_p=W_q$. Readily, $\alpha(x)\equiv Ax\pmod{\mm^2}$ for some
$A\in\mathrm{GL}_n(\mathbb R)$, where $x=(x_1,\ldots,x_n)^{\mathsf T}$.
Comparing quadratic terms of $\alpha(W_p)$ and $W_q$ gives
\[
 A^{\mathsf T}\operatorname{diag}(I_p,-I_{n-p})A
   =\operatorname{diag}(I_q,-I_{n-q}).
\]
Sylvester's law of inertia forces $p=q$.

In fact, these are all the right-equivalence classes of rigid
potentials on $Q$. Let $\Phi\in\pi(\mm^2)$ be rigid.
By Lemma~\ref{lem:quadratic-splitting}, after a right equivalence
we may assume
\[
 \Phi=\sum_{i=1}^{M}\pi(\lambda_i x_i^2)+\Phi_{\mathrm{red}},
\]
where $0\leq M\leq n$, each $\lambda_i\in\mathbb R^\times$,
and $\Phi_{\mathrm{red}}\in\pi(\mm^3)$ involves none of the
$x_i$ for $1\leq i\leq M$. Since $\Phi$ is rigid, the
path-length-one parts of its cyclic derivatives span $\mm/\mm^2$.
Thus $M=n$, so $\Phi_{\mathrm{red}}=0$.
By Sylvester's law of inertia, $\Phi$ is right equivalent to
some $W_p$. Consequently, there are exactly $n+1$
right-equivalence classes of rigid potentials on $Q$ over $\mathbb R$.
\end{example}

\section{Determinacy bounds}\label{sec:weighted}
The finite-determinacy criterion guarantees that a trace-finite
potential is determined up to right equivalence by some finite jet,
but does not specify how high that jet must be. In this section, we
seek \emph{determinacy bounds} for weighted-homogeneous potentials:
explicit integers $N$ such that agreement of the ordinary $N$-jets
implies right equivalence.

If $Q_1=\varnothing$, then $\pi(\mm^2)=0$ and the zero potential is
$0$-determined. We therefore assume $Q_1\ne\varnothing$.
Throughout Section~\ref{sec:weighted} and Section~\ref{sec:graded-invariants}, a weight system
$\zeta=(\zeta_a)_{a\in Q_1}$ assigns a strictly positive integer
to each arrow. Fix such a system, and write
\[
 \underline\zeta=\min_{a\in Q_1}\zeta_a,\qquad
 \bar\zeta=\max_{a\in Q_1}\zeta_a.
\]
For a path $u=a_1\cdots a_r$, put
$d_\zeta(u)=\zeta_{a_1}+\cdots+\zeta_{a_r}$, and set
$d_\zeta(e_i)=0$. Then
\[
 \underline\zeta|u|\leq d_\zeta(u)\leq\bar\zeta|u|.
\]
Since $Q$ is finite and the weights are positive, there are
only finitely many paths of any given weight. Consequently,
every $f\in\mF$ has a unique expansion
\[
f=\sum_{n\geq0} f_n,
\]
where $f_n$ is a linear combination of paths of $\zeta$-weight
$n$. The inequalities above show that
this expansion converges in the $\mm$-adic topology and that
the weight filtration defines the same topology.
Here the \emph{ordinary order} of a nonzero series or potential is
the least path length of a nonzero component, whereas \emph{weight}
refers to the fixed system $\zeta$.

The weight decomposition also descends to the trace space
$\operatorname{Tr}(\mF)$. Indeed, if $f=\sum_i f_i$ and
$g=\sum_j g_j$, then the weight-$n$ component of $[f,g]$ is
$\sum_{i+j=n}[f_i,g_j]$. Thus each weight projection
preserves $[\mF,\mF]$ and induces a projection on
$\operatorname{Tr}(\mF)$. Since $\operatorname{Tr}(\mF)$ is PC,
every potential has a unique convergent decomposition
into weight components.

\begin{definition}
A potential $\Phi$ on $Q$ is called
\emph{$\zeta$-homogeneous of weight $n$} if it is a linear
combination of cyclic classes of cycles of $\zeta$-weight $n$.
For $\Phi\ne0$, its weight $n$ is uniquely determined by the
weight decomposition above. We use the same notation as for paths
and write $d_\zeta(\Phi)=n$.
By convention, $d_\zeta(0)=-\infty$.
A potential is called \emph{weighted-homogeneous} if it is
homogeneous with respect to some weight system.
\end{definition}

\begin{remark}\label{rem:weighted-quasihomogeneous}
Recall that a potential $\Phi$ is \emph{quasi-homogeneous} if
$\Phi\in\pi(\mJ(\Phi))$, in the terminology of Hua and Zhou
\cite[Definition~2.6]{HZ21}. Every $\zeta$-homogeneous
potential is quasi-homogeneous. The assertion is clear for
$\Phi=0$. For $\Phi\ne0$, the weighted cyclic Euler identity gives
\[
 d_\zeta(\Phi)\,\Phi
   =\pi\left(\sum_{a\in Q_1}\zeta_a\,a\,\partial_a\Phi\right).
\]
In particular, the class of $\Phi$ in
$\operatorname{Tr}(\Lambda(\Phi))$ vanishes. In the one-vertex case,
Hua and Zhou established a converse up to right equivalence for
Jacobi-finite potentials of order at least three over an algebraically
closed field of characteristic zero \cite[Theorem~4.1]{HZ21}.
\end{remark}

For a $\zeta$-homogeneous potential $\Phi$, the cyclic derivatives
are $\zeta$-homogeneous. Hence the closed Jacobian ideal is
$\zeta$-homogeneous. Let $\mF_d$ denote the subspace spanned by paths
of $\zeta$-weight $d$, and put
$\Lambda(\Phi)_d=(\mF_d+\mJ(\Phi))/\mJ(\Phi)$ and
$\operatorname{Tr}(\Lambda(\Phi))_d
=\pi_{\Lambda(\Phi)}(\Lambda(\Phi)_d)$.
The Jacobian algebra and its trace space inherit completed $\zeta$-weight
gradings
\[
 \Lambda(\Phi)=\prod_{d\geq0}\Lambda(\Phi)_d,\qquad
 \operatorname{Tr}(\Lambda(\Phi))
 =\prod_{d\geq0}\operatorname{Tr}(\Lambda(\Phi))_d.
\]
Each $\zeta$-weight component is finite dimensional. Define the Hilbert series of the graded Jacobian algebra and its graded trace space by
\[
H_{\zeta,\Phi}(t)=\sum_{d\geq0}\dim_\Bbbk\Lambda(\Phi)_d\,t^d,
\qquad
T_{\zeta,\Phi}(t)=\sum_{d\geq0}\dim_\Bbbk\operatorname{Tr}(\Lambda(\Phi))_d\,t^d.
\]
If $\Phi$ is trace-finite, then $T_{\zeta,\Phi}(t)$ is
a polynomial; set
\begin{equation}\label{eq:top-weight}
 a_\zeta^{\mathrm{tr}}(\Phi)=\deg T_{\zeta,\Phi}(t).
\end{equation}
Vertex classes occur in weight zero, so the polynomial degree is well defined
under the trace-finite hypothesis. The notation is motivated by
the commutative case, where the Milnor algebra of a
weighted-homogeneous isolated hypersurface singularity is a
graded Artinian Gorenstein algebra, whose $a$-invariant equals
its highest nonzero weight (Remark~\ref{rem:commutative-weighted}).
The superscript $\mathrm{tr}$ indicates that our quantity records
the highest nonzero weight of the \emph{trace space}.

\begin{theorem}\label{thm:weighted-general}
Let $\Phi\in\pi(\mm^2)$ be $\zeta$-homogeneous and trace-finite.
Then $\Phi$ is $N$-determined, where
\begin{equation}\label{eq:weighted-general-bound}
 N=\left\lfloor\frac{\max\{d_\zeta(\Phi),a_\zeta^{\mathrm{tr}}(\Phi)\}}
                         {\underline\zeta}\right\rfloor.
\end{equation}
\end{theorem}

To prove the theorem, we need the following lemma.

\begin{lemma}\label{lem:trace-weight-comparison}
Let $\Phi\in\pi(\mm^2)$ be $\zeta$-homogeneous and trace-finite.
Then every loop $a\in Q_1$ satisfies
$\zeta_a<d_\zeta(\Phi)$, and there are natural isomorphisms
of weight components
\[
 \operatorname{Tr}(\Lambda(\Phi))_d\cong
 \left(\frac{\pi(\mm^2)}{\pi(\mm\mJ(\Phi))}\right)_d
 \qquad(d\geq\max\{1,d_\zeta(\Phi)\}).
\]
\end{lemma}

\begin{proof}
If $\Phi=0$, trace-finiteness forces $Q$ to be acyclic:
an oriented cycle would yield infinitely many independent cyclic
classes in $\operatorname{Tr}(\mF)$. Hence the claimed isomorphism holds in positive weights. We may therefore assume $\Phi\ne0$.
Put $M=d_\zeta(\Phi)$. Suppose a loop $a$ had $\zeta_a\geq M$.
As $\Phi\in\pi(\mm^2)$, the arrow $a$ cannot occur in any cyclic
monomial of $\Phi$. Sending $a$ to $t$, the other arrows to zero,
the vertex idempotent corresponding to $s(a)=t(a)$ to $1$, and the
other vertex idempotents to zero gives a continuous surjection
$\Lambda(\Phi)\to\Bbbk[[t]]$. This induces a surjection on trace spaces,
contradicting finite dimensionality. Thus every loop has weight
smaller than $M$, so $\pi(\mm)_d=\pi(\mm^2)_d$ and hence
$(\pi(\mF)/\pi(\mm^2))_d=0$ for all $d\geq M$. In addition, the quotient
$\pi(\mJ(\Phi))/\pi(\mm\mJ(\Phi))$ is spanned by the cyclic
classes of $\partial_a\Phi$ for loops $a$. These classes have
weights $M-\zeta_a<M$, so
$(\pi(\mJ(\Phi))/\pi(\mm\mJ(\Phi)))_d=0$ for $d\geq M$.
Now consider the canonical map
\[
\frac{\pi(\mm^2)}{\pi(\mm\mJ(\Phi))}
\longrightarrow
\frac{\pi(\mF)}{\pi(\mJ(\Phi))}
=\operatorname{Tr}(\Lambda(\Phi)).
\]
Its kernel is a graded subspace of
$\pi(\mJ(\Phi))/\pi(\mm\mJ(\Phi))$, and its cokernel is a quotient
of $\pi(\mF)/\pi(\mm^2)$. The result follows directly.
\end{proof}

\begin{proof}[Proof of Theorem~\ref{thm:weighted-general}]
If $\Phi=0$, trace-finiteness forces $Q$ to be acyclic.
Thus $\pi(\mm^2)=0$, and the assertion holds with $N=0$.
Hence assume $\Phi\ne0$, and put
$M=d_\zeta(\Phi)$ and $D=\max\{M,a_\zeta^{\mathrm{tr}}(\Phi)\}$.
Every path of length at least $N+1$ has $\zeta$-weight at least
$\underline\zeta(N+1)>D$. Thus, if $\Psi^{[N]}=\Phi^{[N]}$, then $\Psi-\Phi$ has no terms of $\zeta$-weight at most $D$.
It therefore suffices to show that every potential $\Psi$ such that
$\Psi-\Phi$ has no terms of $\zeta$-weight at most $D$ is right equivalent to $\Phi$.

Write $\Psi=\Phi+\Theta^{(0)}$, and let $\Theta_d^{(0)}$ be the nonzero component of $\Theta^{(0)}$ of lowest $\zeta$-weight, where $d>D$. By Lemma~\ref{lem:trace-weight-comparison}, the weight-$d$ component of $\pi(\mm^2)/\pi(\mm\mJ(\Phi))$ vanishes. Hence $\Theta_d^{(0)}\in\pi(\mm\mJ(\Phi))$, so we can write
\[
 \Theta_d^{(0)}=\pi\left(\sum_a u_a\partial_a\Phi\right),
 \qquad u_a\in e_{t(a)}\mm e_{s(a)}.
\]
Taking $\zeta$-homogeneous components replaces $u_a$ by its component
of weight $d-M+\zeta_a$. The arrow assignment
$\delta(a)=u_a$ then defines an $\ell$-derivation satisfying
\[
 \delta(\Phi)=\Theta_d^{(0)}.
\]
The derivation $\delta$ raises $\zeta$-weights by $d-M>0$
and therefore preserves $\mm$. Since the $\zeta$-weight filtration
generates the $\mm$-adic topology, the exponential series
$\exp(-\delta)=\sum_{n\geq0}(-\delta)^n/n!$ converges.
By the iterated Leibniz formula,
$\exp(-\delta)$ is an $\ell$-unital algebra automorphism with inverse
$\exp(\delta)$. Though the map $\mm/\mm^2\to\mm/\mm^2$ induced by
$\exp(-\delta)$ need not be the identity, the cancellation in
Lemma~\ref{lem:correction} still holds, which gives that
$\exp(-\delta)(\Phi+\Theta_d^{(0)})-\Phi$ has lowest $\zeta$-weight at least
$2d-M>d$. Since $\Theta^{(0)}-\Theta_d^{(0)}$ has $\zeta$-weights greater than $d$ and
$\exp(-\delta)$ does not lower $\zeta$-weights, the lowest $\zeta$-weight of
$\exp(-\delta)(\Phi+\Theta^{(0)})-\Phi$ is greater than $d$.

Iterating, at the $n$th step let $d_n$ be the lowest
$\zeta$-weight of $\Theta^{(n)}$. The correcting automorphism $\alpha_n=\exp(-\delta_n)$, with
$\delta_n\in\mDer_\ell(\mF)$ satisfying
$\delta_n(\Phi)=\Theta^{(n)}_{d_n}$, gives
\[
\Phi+\Theta^{(n+1)}=\alpha_n(\Phi+\Theta^{(n)}),
\qquad d_{n+1}>d_n
\]
if the process does not terminate. The same argument as for $\alpha_0$ shows that each $\alpha_n$ is a well-defined $\ell$-unital algebra automorphism of $\mF$. Each correcting automorphism satisfies
\[
\operatorname{ord}(\alpha_n(a)-a)
\geq \left\lceil\frac{\zeta_a+d_n-M}{\bar\zeta}\right\rceil
\qquad(a\in Q_1),
\]
where $\bar\zeta=\max_{a\in Q_1}\zeta_a$. Thus the ordinary depths
of the $\alpha_n$ tend to infinity. By Lemma~\ref{lem:products},
the products $\alpha_n\cdots\alpha_0$
converge to an automorphism $\alpha\in\mAut_\ell(\mF)$, which satisfies
$\alpha(\Psi)=\Phi$.
\end{proof}

\Needspace{12\baselineskip}
\begin{example}\label{ex:weighted-bound-two-terms}
Let $Q$ have one vertex and two loops $x,y$, with
$\zeta(x)=3$ and $\zeta(y)=2$. Consider
\[
 \Phi=\pi(x^2+y^3),\qquad \Psi=\pi(x^2y+y^4).
\]
Their Jacobian algebras have presentations
\[
\Lambda(\Phi)=\Bbbk[[y]]/(y^2),\qquad
\Lambda(\Psi)=\Bbbk\langle\langle x,y\rangle\rangle/
\overline{(xy+yx,\,x^2+4y^3)}.
\]
The algebra $\Lambda(\Phi)$ has basis $1,y$, and its commutator subspace is zero.
By the Gr\"obner--Shirshov basis theory for closed ideals of complete free algebras
\cite[Section~2]{IL21}, a basis of $\Lambda(\Psi)$ consists of
$1,y,\ldots,y^5$ and $x,xy,xy^2$.
Its commutator subspace is spanned by
$xy,xy^2,y^4$. Thus both potentials are Jacobi-finite, with trace Hilbert series
\[
 T_{\zeta,\Phi}(t)=1+t^2,\qquad
 T_{\zeta,\Psi}(t)=1+t^2+t^3+t^4+t^6+t^{10}.
\]
Consequently,
\[
 d_\zeta(\Phi)=6>2=a_\zeta^{\mathrm{tr}}(\Phi),\qquad
 d_\zeta(\Psi)=8<10=a_\zeta^{\mathrm{tr}}(\Psi).
\]
Thus neither term in the maximum of
Theorem~\ref{thm:weighted-general} can be omitted.
\end{example}

\begin{corollary}\label{cor:homogeneous-exact}
Suppose that $\Phi\in\pi(\mm^2)$ is $\zeta$-homogeneous and trace-finite, with $\zeta_a=1$ for each $a\in Q_1$. Then its smallest determinacy order is
\[
 \max\{d_\zeta(\Phi),a_\zeta^{\mathrm{tr}}(\Phi)\}.
\]
\end{corollary}
\begin{proof}
Let $N_0$ be the smallest determinacy order of $\Phi$.
If $\Phi=0$, trace-finiteness forces $Q$ to be acyclic, so its only potential
is zero and $N_0=0$, as asserted. Assume henceforth that $\Phi\ne0$, and put
$M=d_\zeta(\Phi)$, the common path length of the nonzero cyclic terms of $\Phi$.
Theorem~\ref{thm:weighted-general} gives the upper bound. An order smaller than
$M$ cannot determine $\Phi$, since the zero potential has the same jet and is not
right equivalent to $\Phi$. Thus $N_0\geq M$. By Proposition~\ref{prop:fd-necessary},
$\pi(\mm^{N_0+1})\subseteq\pi(\mm\mJ(\Phi))$.
It follows that
$\left(\pi(\mm^2)/\pi(\mm\mJ(\Phi))\right)_d=0$ for all $d>N_0$.
Then, by Lemma~\ref{lem:trace-weight-comparison},
$\operatorname{Tr}(\Lambda(\Phi))_d=0$ for all $d>N_0$.
Thus $N_0\geq a_\zeta^{\mathrm{tr}}(\Phi)$, proving the lower bound.
\end{proof}

\begin{prop}\label{prop:uniform-weight-bound}
For every integer $M\geq2$, there exists an integer
$D=D(\zeta,M)$ such that
\[
 a_\zeta^{\mathrm{tr}}(\Phi)\leq D
\]
for every Jacobi-finite $\zeta$-homogeneous potential
$\Phi\in\pi(\mm^2)_M$.
\end{prop}

\begin{proof}
Let $V=\pi(\mm^2)_M$, which is finite dimensional. For each $N\geq1$, define
\[
 U_N=\{\Phi\in V:\Lambda(\Phi)_n=0\text{ for }N\leq n<N+\bar\zeta\}.
\]
For $\Phi\in U_N$, let $p$ be a path of $\zeta$-weight at least $N$,
and list its arrows in traversal order as $a_1,\ldots,a_\ell$.
Choose the least $j$ such that $\sum_{i=1}^j\zeta_{a_i}\geq N$.
Then
\[
 N\leq\sum_{i=1}^j\zeta_{a_i}<N+\bar\zeta,
\]
so the corresponding subpath has weight in $[N,N+\bar\zeta-1]$ and hence
represents zero in $\Lambda(\Phi)$. Thus $p$ itself represents zero,
and consequently
\[
 U_N=\{\Phi\in V:\Lambda(\Phi)_n=0\text{ for all }n\geq N\}.
\]
In particular, $U_N\subseteq U_{N+1}$, and $\bigcup_{N\geq1}U_N$
is precisely the Jacobi-finite locus in $V$.

For each $n\geq1$, the condition $\Lambda(\Phi)_n=0$ defines a
Zariski-open subset of $V$. Indeed, the weight-$n$ component of
$\mJ(\Phi)$ is spanned by the elements $u(\partial_a\Phi)v$ of
weight $n$, where $u$ and $v$ range over paths. The corresponding
matrix has entries linear in the coefficients of $\Phi$, and the
vanishing of $\Lambda(\Phi)_n$ is a maximal-rank condition. Since
$U_N$ is defined by finitely many such conditions, it is Zariski open.
The affine space $V$ is Noetherian, so the increasing sequence
$U_1\subseteq U_2\subseteq\cdots$ stabilizes, say at $U_{N_0}$.
Every Jacobi-finite $\Phi\in V$ therefore satisfies
$\Lambda(\Phi)_n=0$ for all $n\geq N_0$, and hence
$\operatorname{Tr}(\Lambda(\Phi))_n=0$ for all $n\geq N_0$.
We may take $D=N_0-1$.
\end{proof}

\section{Graded invariants under right equivalence}\label{sec:graded-invariants}
The determinacy bounds established in the previous section involve
weighted-homogeneous potentials and the graded data of their Jacobian
algebras and trace spaces. This raises the question of which of these
quantities are preserved under right equivalence. We investigate the
$\zeta$-weight of a potential and the Hilbert series of its Jacobian
algebra and trace space, beginning with the analogous results for
weighted-homogeneous isolated hypersurface singularities.

\begin{remark}\label{rem:commutative-weighted}
Let $R=\Bbbk[[x_1,\ldots,x_n]]$ be equipped with a
weight system $\zeta=(\zeta_1,\ldots,\zeta_n)$, where
$x_i$ has $\zeta$-weight $\zeta_i$. Let
$f\in(x_1,\ldots,x_n)^2$ be $\zeta$-homogeneous of weight $M$
with an isolated critical point. Put
$J_f=(\partial_{x_1}f,\ldots,\partial_{x_n}f)$ and let $\mathcal M(f)=R/J_f$ be the Milnor algebra of $f$.
Since the partial derivatives of $f$ form a regular sequence,
the well-known Hilbert series formula for graded complete intersections
(see, e.g., \cite[Exercises~10.12--10.13]{Ei95}) gives
\[
 H_{\mathcal M(f)}(t)=\prod_{i=1}^n\frac{1-t^{M-\zeta_i}}{1-t^{\zeta_i}}.
\]
This yields the classical Milnor--Orlik formula for the Milnor number
$\mu(f)=\dim_\Bbbk \mathcal M(f)$ \cite{GLS25, MO70}, together
with the formula for the $a$-invariant
$a(f)=\deg H_{\mathcal M(f)}(t)$:
\[
 \mu(f)=\prod_{i=1}^n\frac{M-\zeta_i}{\zeta_i},
 \qquad
 a(f)=nM-2\sum_{i=1}^n \zeta_i.
\]
If $g$ is a $\zeta$-homogeneous series of weight $M'$ that is
right equivalent to $f$, the chain rule gives
$\mu(g)=\mu(f)$ \cite[Section~3.1.1]{GLS25}.
Isolatedness forces each variable to occur, so
$M,M'>\max_i \zeta_i$. The Milnor--Orlik product is strictly increasing as a function of $M$ on this range; hence $M=M'$.
Consequently, $H_{\mathcal M(f)}(t)=H_{\mathcal M(g)}(t)$ and $a(f)=a(g)$.

\end{remark}

For a $\zeta$-homogeneous potential $\Phi\in\pi(\mm^2)$, put
$A=\Lambda(\Phi)$, with augmentation ideal
$A_+=\mm/\mJ(\Phi)$. The quotient $A_+/A_+^2$ is called the
\emph{space of indecomposables} of $A$. It has the following natural
identification as a graded vector space:
\begin{equation}\label{eq:indec-identification}
 \frac{A_+}{A_+^2}\cong
 \frac{\mm}{\mm^2+\mJ(\Phi)}.
\end{equation}
When $\Phi$ is reduced (i.e., $\Phi\in\pi(\mm^3)$),
$A_+/A_+^2\cong\mm/\mm^2$.
Although $A=\prod_{d\geq0}A_d$ is complete with respect to
its $\zeta$-weight grading, each $A_+^r/A_+^{r+1}$ is
finite dimensional, being a quotient of $\mm^r/\mm^{r+1}$.
Its weight decomposition is therefore a finite direct sum.
Thus the associated graded algebra with respect to the
augmentation-ideal filtration is the algebraic direct sum
\[
 \operatorname{gr}(A)=\bigoplus_{r\geq0}\frac{A_+^r}{A_+^{r+1}}
 =\bigoplus_{r,d\geq0}\left(A_+^r/A_+^{r+1}\right)_d.
\]
Multiplication sends bidegrees $(r,d)$ and $(q,e)$ to
$(r+q,d+e)$. Consequently, $\operatorname{gr}(A)$ is a
bigraded algebra, with augmentation-ideal degree $r$ and
$\zeta$-weight $d$.
The trace space inherits the filtration by the subspaces
$\pi_A(A_+^r)$, and its associated graded vector space is
\[
 \operatorname{gr}(\operatorname{Tr}(A))
 =\bigoplus_{r\geq0}
 \frac{\pi_A(A_+^r)}{\pi_A(A_+^{r+1})}
 =\bigoplus_{r,d\geq0}
 \left(\frac{\pi_A(A_+^r)}{\pi_A(A_+^{r+1})}\right)_d.
\]
There is a natural surjection of bigraded vector spaces
\[
 \operatorname{Tr}(\operatorname{gr}(A))
 \twoheadrightarrow
 \operatorname{gr}(\operatorname{Tr}(A)),
\]
which need not be an isomorphism.

\begin{theorem}\label{thm:graded-linear-part}
Let $\Phi,\Psi\in\pi(\mm^2)$ be $\zeta$-homogeneous potentials,
and suppose that $\alpha\in\mAut_\ell(\mF)$ satisfies
$\alpha(\Phi)=\Psi$. Assume that the map induced by $\alpha$
\begin{equation}\label{eq:indec-map}
 \frac{\mm}{\mm^2+\mJ(\Phi)}
 \longrightarrow
 \frac{\mm}{\mm^2+\mJ(\Psi)}
\end{equation}
preserves $\zeta$-weights.
Then $\alpha$ induces an isomorphism
\[
 \operatorname{gr}(\Lambda(\Phi))
 \cong\operatorname{gr}(\Lambda(\Psi))
\]
of bigraded algebras. Moreover,
\[
 d_\zeta(\Phi)=d_\zeta(\Psi),\qquad
 H_{\zeta,\Phi}(t)=H_{\zeta,\Psi}(t),\qquad
 T_{\zeta,\Phi}(t)=T_{\zeta,\Psi}(t).
\]
In particular, if $\Phi$ and $\Psi$ are trace-finite, then
$a_\zeta^{\mathrm{tr}}(\Phi)=a_\zeta^{\mathrm{tr}}(\Psi)$.
\end{theorem}

\begin{proof}
Put $A=\Lambda(\Phi)$ and $B=\Lambda(\Psi)$.
By Lemma~\ref{lem:jac-invariance}, $\alpha$ induces a filtered
isomorphism $\overline\alpha:A\to B$.
Under the natural identifications~\eqref{eq:indec-identification} for $\Phi$ and $\Psi$,
the induced map $A_+/A_+^2\to B_+/B_+^2$ preserves
$\zeta$-weights. Since $A_+^r/A_+^{r+1}$ is generated
by products of $r$ elements of $A_+/A_+^2$, the map
$\operatorname{gr}(\overline\alpha):\operatorname{gr}(A)\to\operatorname{gr}(B)$
is an isomorphism of bigraded algebras.
Consequently,
\[
H_{\zeta,\Phi}(t)
=\sum_{d\geq0}\sum_{r\geq0}\dim_\Bbbk(A_+^r/A_+^{r+1})_d\,t^d
=\sum_{d\geq0}\sum_{r\geq0}\dim_\Bbbk(B_+^r/B_+^{r+1})_d\,t^d
=H_{\zeta,\Psi}(t).
\]

The induced map
$\operatorname{Tr}(\operatorname{gr}(A))\to
\operatorname{Tr}(\operatorname{gr}(B))$ preserves both gradings.
The natural square
\[
\begin{array}{ccc}
 \operatorname{Tr}(\operatorname{gr}(A))
 &\xrightarrow{\operatorname{Tr}(\operatorname{gr}(\overline\alpha))}&
 \operatorname{Tr}(\operatorname{gr}(B))\\
 \big\downarrow&&\big\downarrow\\
 \operatorname{gr}(\operatorname{Tr}(A))
 &\xrightarrow{\operatorname{gr}(\operatorname{Tr}(\overline\alpha))}&
 \operatorname{gr}(\operatorname{Tr}(B))
\end{array}
\]
commutes. Since the vertical maps are surjective, the lower
horizontal isomorphism also preserves the bigrading.
Similarly, we obtain $T_{\zeta,\Phi}(t)=T_{\zeta,\Psi}(t)$.

It remains to compare the weights of the potentials.
If $\Phi=0$, then $\Psi=0$, so the assertion follows
from our convention. Otherwise, put
\[
 V=\mm/\mm^2,\qquad
 K_\Phi=(\mJ(\Phi)+\mm^2)/\mm^2\subseteq V,
\]
and define $K_\Psi$ similarly. Both are graded by $\zeta$-weights,
and the weight-preserving isomorphism $V/K_\Phi\cong
V/K_\Psi$ implies
$\dim(K_\Phi)_n=\dim(K_\Psi)_n$ for every integer $n$.

If $K_\Phi=0$, then $K_\Psi=0$. The map
$V/K_\Phi\to V/K_\Psi$ is therefore the linear part
$\alpha_1:V\to V$ of $\alpha$, which preserves
$\zeta$-weights by hypothesis. Since $\alpha$ preserves
ordinary order, let
$r=\operatorname{ord}(\Phi)=\operatorname{ord}(\Psi)$.
The nonzero ordinary length-$r$ components of $\Phi$ and
$\Psi$ are related by the map induced by $\alpha_1$ on
cyclic classes of paths of length $r$. This map preserves
$\zeta$-weights, so $d_\zeta(\Phi)=d_\zeta(\Psi)$.

Suppose now that $K_\Phi\ne0$. Let $M=d_\zeta(\Phi)$
and write $\Phi_2$ for the ordinary quadratic part of $\Phi$.
For each weight $n$, cyclic differentiation of $\Phi_2$
with respect to arrows of weight $n$ defines a linear map
$D_n:V_n^*\to V_{M-n}$, whose image is
$(K_\Phi)_{M-n}$.
Since $\pi(ab)=\pi(ba)$ for every quadratic cycle,
the matrices of $D_n$ and $D_{M-n}$ are transposes.
Thus
\[
 \dim(K_\Phi)_{M-n}=\operatorname{rank}D_n
 =\operatorname{rank}D_{M-n}=\dim(K_\Phi)_n.
\]
Consequently,
\[
 \sum_n n\dim(K_\Phi)_n
 =\sum_n(M-n)\dim(K_\Phi)_n
 =M\dim K_\Phi-\sum_n n\dim(K_\Phi)_n,
\]
and hence
\[
 d_\zeta(\Phi)=M=\frac{2\sum_n n\dim(K_\Phi)_n}{\dim K_\Phi}.
\]
The same formula holds for $\Psi$. Hence
$d_\zeta(\Phi)=d_\zeta(\Psi)$, because
$\dim(K_\Phi)_n=\dim(K_\Psi)_n$ for every integer $n$.
\end{proof}

\Needspace{12\baselineskip}
\begin{example}
The hypothesis of Theorem~\ref{thm:graded-linear-part} ensures that the
associated graded map of the induced isomorphism of Jacobian algebras
preserves $\zeta$-weights, but the isomorphism itself need not. Indeed, let the quiver $Q$, the weight system
$\zeta$, and the potential $\Psi$ be as in
Example~\ref{ex:weighted-bound-two-terms}, and put $u=1+y$.
The inner automorphism $\alpha(f)=ufu^{-1}$ fixes $\Psi$ and
induces the identity on $\mm/\mm^2$, so the hypothesis is
satisfied. However, the induced automorphism $\overline\alpha$
of $\Lambda(\Psi)$ satisfies
\[
\overline\alpha(x)=x-2xy+2xy^2,
\]
which is not $\zeta$-homogeneous. Here $x$ and $y$ also denote
their images in $\Lambda(\Psi)$.
\end{example}

\begin{remark}\label{rem:parallel-arrows-equal-weights}
If parallel arrows of $Q$ have the same $\zeta$-weight, then
for any $\zeta$-homogeneous potentials
$\Phi,\Psi\in\pi(\mm^2)$ and every
$\alpha\in\mAut_\ell(\mF)$ with $\alpha(\Phi)=\Psi$,
the map $\mm/\mm^2\to\mm/\mm^2$ induced by $\alpha$
preserves $\zeta$-weights, and hence so does the
map~\eqref{eq:indec-map}.
\end{remark}

\begin{prop}\label{prop:graded-invariants}
Suppose that $Q$ has one vertex.
Let $\Phi,\Psi\in\pi(\mm^2)$ be $\zeta$-homogeneous
potentials such that the abelianizations of
$\Lambda(\Phi)$ and $\Lambda(\Psi)$ are finite dimensional.
Then, for every $\alpha\in\mAut_\ell(\mF)$ satisfying
$\alpha(\Phi)=\Psi$, the induced map~\eqref{eq:indec-map}
preserves $\zeta$-weights.
\end{prop}

\begin{proof}
We may assume $Q_1\ne\varnothing$.
Fix $\alpha\in\mAut_\ell(\mF)$ with $\alpha(\Phi)=\Psi$,
and put $A=\Lambda(\Phi)$ and $B=\Lambda(\Psi)$.
Let $a_1,\ldots,a_n$ be the loops.
Put $R=\Bbbk[[x_1,\ldots,x_n]]$, and let
$\rho:\mF\to R$ be the abelianization map, with
$\rho(a_i)=x_i$. Let $f,g\in R$ be the
commutativizations of $\Phi,\Psi$, respectively.
Their Milnor algebras $\mathcal M(f)$ and $\mathcal M(g)$
(see Remark~\ref{rem:commutative-weighted}) are the
abelianizations of $A$ and $B$, respectively, and hence
are finite dimensional. The formal coordinate
change of $R$ induced by $\alpha$ takes $f$ to $g$.
Thus, as discussed in Remark~\ref{rem:commutative-weighted},
$f$ and $g$ have the same weight, and hence so do $\Phi$ and
$\Psi$; denote their common weight by $M$.

Let $E_\zeta$ be the $\zeta$-Euler derivation of $\mF$,
defined by $E_\zeta(a_i)=\zeta_{a_i}a_i$, and put
\[
 \delta=\alpha^{-1}E_\zeta\alpha-E_\zeta.
\]
Since $\Phi$ and $\Psi$ have the same $\zeta$-weight $M$,
we have $\delta(\Phi)=0$ in $\operatorname{Tr}(\mF)$.
Under abelianization, $\delta$ induces a derivation
$D=\sum_i h_i\partial_{x_i}$ of $R$ with
\[
 \sum_i h_i\partial_{x_i}f=0.
\]
Because $\mathcal M(f)$ is finite dimensional, the partial derivatives
of $f$ form a regular sequence in $R$. The Koszul relations
generate their syzygies
\cite[Appendix~B, Sections~B.6--B.8]{GLS25}, so
$h_i\in (\partial_{x_1}f,\ldots,\partial_{x_n}f)$ for every $i$.
As $\ker\rho\subseteq\mm^2$ and
$\rho(\mJ(\Phi))=(\partial_{x_1}f,\ldots,\partial_{x_n}f)$,
we obtain
\[
 \delta(a_i)\in\mJ(\Phi)+\mm^2
 \qquad(1\leq i\leq n).
\]
Together with the identity
\[
 E_\zeta(\alpha(u))-\alpha(E_\zeta(u))
 =\alpha(\delta(u)),
\]
this shows that the map
$\mm/(\mm^2+\mJ(\Phi))\to\mm/(\mm^2+\mJ(\Psi))$
induced by $\alpha$ commutes with the induced
$\zeta$-Euler operators and hence preserves
$\zeta$-weights, since $\operatorname{char}\Bbbk=0$.
\end{proof}

For quiver potentials, Theorem~\ref{thm:graded-linear-part} and
Proposition~\ref{prop:graded-invariants}, together with
Remark~\ref{rem:parallel-arrows-equal-weights}, give invariance
of the $\zeta$-weight and of both Hilbert series under two natural
hypotheses. Example~\ref{ex:two-weighted-degrees} shows that the
$\zeta$-weight and Jacobian Hilbert series may change for arbitrary
arrow weights. An example in Appendix~\ref{app:trace-counterexample}
shows that the trace Hilbert series and its degree may change
even between right-equivalent Jacobi-finite
$\zeta$-homogeneous potentials of the same $\zeta$-weight.

\begin{example}\label{ex:two-weighted-degrees}\normalfont
Let $Q$ have vertices $1,2,3$ and arrows
$a,a':1\to2$, $b:2\to3$, $c:3\to1$, with
$\zeta_a=\zeta_b=\zeta_c=1$ and $\zeta_{a'}=2$.
Set
\[
 \Phi_1=\pi(cba),\qquad \Phi_2=\pi(cba').
\]
Both are $\zeta$-homogeneous, with $d_\zeta(\Phi_1)=3$ and
$d_\zeta(\Phi_2)=4$.
Swapping $a,a'$ and fixing $b,c$ defines a right equivalence
$\alpha\in\mAut_\ell(\mF)$ with $\alpha(\Phi_1)=\Phi_2$.
The Jacobian algebras are
\[
 \Lambda(\Phi_1)=\Bbbk\langle\langle Q\rangle\rangle/\overline{(cb,ac,ba)},\qquad
 \Lambda(\Phi_2)=\Bbbk\langle\langle Q\rangle\rangle/\overline{(cb,a'c,ba')}.
\]
It is easy to see that they are both Jacobi-finite, with
\[
H_{\zeta,\Phi_1}(t)=3+3t+t^2+2t^3+t^4,\qquad
H_{\zeta,\Phi_2}(t)=3+3t+3t^2+t^3.
\]
So $\Phi_1$ and $\Phi_2$ have different $\zeta$-weights and Jacobian
Hilbert series. However, both are readily seen to be rigid, hence
$T_{\zeta,\Phi_1}(t)=T_{\zeta,\Phi_2}(t)=3$ and
$a_\zeta^{\mathrm{tr}}(\Phi_1)=a_\zeta^{\mathrm{tr}}(\Phi_2)=0$.
\end{example}

For any $\zeta$-homogeneous potential $\Phi$, define
\[
 \mathcal D_\zeta(\Phi)=
 \{d_\zeta(\Psi):\Psi\sim_{\mathrm r}\Phi
       \text{ and }\Psi\text{ is }\zeta\text{-homogeneous}\},
\]
where $\sim_{\mathrm r}$ denotes right equivalence. If $\Phi$ is trace-finite, define also
\[
 \mathcal A_\zeta^{\mathrm{tr}}(\Phi)=
 \{a_\zeta^{\mathrm{tr}}(\Psi):\Psi\sim_{\mathrm r}\Phi
       \text{ and }\Psi\text{ is }\zeta\text{-homogeneous}\}.
\]
These sets, whenever defined, depend only on the right-equivalence
class of $\Phi$.

\begin{prop}\label{prop:graded-value-sets}
For every $\zeta$-homogeneous potential $\Phi$, the set
$\mathcal D_\zeta(\Phi)$ is finite. If $\Phi$ is trace-finite,
then $\mathcal A_\zeta^{\mathrm{tr}}(\Phi)$ is also finite.
\end{prop}
\begin{proof}
We may assume $Q_1\neq\varnothing$ and $\Phi\neq0$.
Right equivalence preserves the ordinary order
$p=\operatorname{ord}(\Phi)$. Every $\zeta$-homogeneous
$\Psi\sim_{\mathrm r}\Phi$ has a nonzero component of ordinary
length $p$, so its $\zeta$-weight is the weight of some cyclic
path of length $p$. Thus $\mathcal D_\zeta(\Phi)$ is finite.

Suppose now that $\Phi$ is trace-finite. By
Theorem~\ref{thm:trace-criterion}, $\Phi$ is $N$-determined
for some $N\geq1$, and every right-equivalent potential $\Psi$
is also $N$-determined. Proposition~\ref{prop:fd-necessary} gives
$\pi(\mm^{N+1})\subseteq\pi(\mm\mJ(\Psi))$.
Every cyclic path of $\zeta$-weight greater than
$N\bar\zeta$ has ordinary length greater than $N$, where
$\bar\zeta=\max_{a\in Q_1}\zeta_a$, so
$a_\zeta^{\mathrm{tr}}(\Psi)\leq N\bar\zeta$.
Thus $\mathcal A_\zeta^{\mathrm{tr}}(\Phi)$ is a finite set.
\end{proof}

\begin{remark}\label{rem:weighted-degree-extrema}
In the commutative isolated-singularity setting of
Remark~\ref{rem:commutative-weighted}, the corresponding sets
are singletons. Example~\ref{ex:two-weighted-degrees} and the counterexample
in Appendix~\ref{app:trace-counterexample} show that this
need not hold for quiver potentials. For trace-finite $\Phi$, it would be
interesting to understand the relationship between
$\mathcal D_\zeta(\Phi)$ and $\mathcal A_\zeta^{\mathrm{tr}}(\Phi)$.
In particular, their respective minima and maxima are
right-equivalence invariants, and one may ask how these four
invariants are related.
\end{remark}

\appendix
\section{A counterexample to graded trace invariance}\label{app:trace-counterexample}

The aim of this appendix is to construct two right-equivalent,
Jacobi-finite $\zeta$-homogeneous potentials of the same
$\zeta$-weight whose trace Hilbert series differ and have
different degrees.

\begin{searchstrategy}
Our strategy is to first find a systematic
way to construct right-equivalent $\zeta$-homogeneous potentials
for which the induced map~\eqref{eq:indec-map} fails to preserve
$\zeta$-weights, and then to search among such pairs for examples
with different trace Hilbert series.
Remark~\ref{rem:parallel-arrows-equal-weights} suggests
starting with a quiver without parallel arrows, splitting
selected arrows into parallel copies, and assigning different
positive weights to those copies.
Proposition~\ref{prop:graded-invariants} indicates that loops do not help produce the desired failure of invariance, so the starting quiver should be loop-free. We also seek a sufficiently symmetric quiver carrying a
one-parameter family of potentials that are generically Jacobi-finite
and admit a convenient matrix expression. By splitting suitable arrows, we can lift this
expression to a block matrix in which the scalar parameter is
replaced by a diagonal matrix, yielding a multi-parameter family
of potentials on the split quiver.
\end{searchstrategy}

The quiver $Q^{(2)}$ from the work of Geiss, Labardini-Fragoso, and Schröer~\cite[Section~9.9]{GLS16} provides such a starting point: it is loop-free, has no parallel arrows, is sufficiently symmetric, and carries the desired one-parameter family of potentials $W_t^{(2)}$. It is obtained by
mutation from the adjacency quiver of the usual tetrahedral
triangulation of the sphere with four punctures. We draw it in the arrangement used in
\cite[Section~9.9]{GLS16}; the arrow names have been changed to avoid
conflicts with notation in the main text. As in \cite[Section~9.9]{GLS16},
the quiver is drawn by cutting the cylinder open, so the two displayed
copies of vertex~$1$ are identified.
\begin{center}
\begin{tikzpicture}[
 >=Stealth,
 every node/.style={font=\small},
 edge/.style={->,line width=.48pt},
 scale=.90,transform shape
]
\draw[densely dashed] (-4.95,.95) -- (-4.95,-.95);
\draw[densely dashed] (4.95,.95) -- (4.95,-.95);
\node (v1l) at (-4.45,0) {$1$};
\node (v2)  at (0,0) {$2$};
\node (v1r) at (4.45,0) {$1$};
\node (v3)  at (-2.15,1.85) {$3$};
\node (v4)  at (-2.15,-1.85) {$4$};
\node (v5)  at (2.15,1.85) {$5$};
\node (v6)  at (2.15,-1.85) {$6$};
\draw[edge] (v2) -- node[pos=.53,sloped,fill=white,inner sep=.8pt] {$a_+$} (v3);
\draw[edge] (v3) -- node[pos=.51,sloped,fill=white,inner sep=.8pt] {$p_+$} (v1l);
\draw[edge] (v2) -- node[pos=.53,sloped,fill=white,inner sep=.8pt] {$a_-$} (v4);
\draw[edge] (v4) -- node[pos=.51,sloped,fill=white,inner sep=.8pt] {$p_-$} (v1l);
\draw[edge] (v1r) -- node[pos=.51,sloped,fill=white,inner sep=.8pt] {$b_+$} (v5);
\draw[edge] (v5) -- node[pos=.53,sloped,fill=white,inner sep=.8pt] {$q_+$} (v2);
\draw[edge] (v1r) -- node[pos=.51,sloped,fill=white,inner sep=.8pt] {$b_-$} (v6);
\draw[edge] (v6) -- node[pos=.53,sloped,fill=white,inner sep=.8pt] {$q_-$} (v2);
\end{tikzpicture}
\end{center}
Put
\[
 u^\pm=p_\pm a^\pm,\qquad v^\pm=q_\pm b^\pm.
\]
\begin{samepage}
For $t\in\Bbbk$, the potential $W_t^{(2)}$ on $Q^{(2)}$ from \cite[Section~9.9]{GLS16} can be written as
\[
 W_t^{(2)}=\pi\bigl((u^*)^{\mathsf T}C_tv^*\bigr)
 =\pi\bigl(tu^+v^++u^+v^-+u^-v^++u^-v^-\bigr),
\]
where
\[
 u^*=\begin{pmatrix}u^+\\u^-\end{pmatrix},\qquad
 v^*=\begin{pmatrix}v^+\\v^-\end{pmatrix},\qquad
 C_t=\begin{pmatrix}t&1\\[2pt]1&1\end{pmatrix}.
\]
It is known from \cite[Lemma~9.20(a)]{GLS16} that $W_t^{(2)}$ is
Jacobi-finite for $t\ne0,1$.
\end{samepage}

We now split the arrows $a^+,b^+,a^-,b^-$ into three parallel copies,
while leaving $p_\pm,q_\pm$ unchanged. The resulting split quiver is
drawn below in the same cylindrical layout.
\begin{center}
\begin{tikzpicture}[
 >=Stealth,
 every node/.style={font=\small},
 edge/.style={->,line width=.45pt},
 scale=.94,transform shape
]
\draw[densely dashed] (-6.15,1.05) -- (-6.15,-1.05);
\draw[densely dashed] (6.15,1.05) -- (6.15,-1.05);
\node (v1l) at (-5.55,0) {$1$};
\node (v2)  at (0,0) {$2$};
\node (v1r) at (5.55,0) {$1$};
\node (v3)  at (-3.72,1.86) {$3$};
\node (v4)  at (-3.72,-1.86) {$4$};
\node (v5)  at (2.42,1.86) {$5$};
\node (v6)  at (2.42,-1.86) {$6$};
\draw[edge] (v3) -- node[pos=.52,sloped,fill=white,inner sep=.8pt] {$p_+$} (v1l);
\draw[edge] (v4) -- node[pos=.52,sloped,fill=white,inner sep=.8pt] {$p_-$} (v1l);
\draw[edge] (v5) -- node[pos=.52,sloped,fill=white,inner sep=.8pt] {$q_+$} (v2);
\draw[edge] (v6) -- node[pos=.52,sloped,fill=white,inner sep=.8pt] {$q_-$} (v2);
\draw[edge] (v2) to[bend left=18] node[pos=.53,sloped,fill=white,inner sep=.6pt] {$a^+_1$} (v3);
\draw[edge] (v2) -- node[pos=.55,sloped,fill=white,inner sep=.6pt] {$a^+_2$} (v3);
\draw[edge] (v2) to[bend right=18] node[pos=.57,sloped,fill=white,inner sep=.6pt] {$a^+_3$} (v3);
\draw[edge] (v2) to[bend right=18] node[pos=.57,sloped,fill=white,inner sep=.6pt] {$a^-_1$} (v4);
\draw[edge] (v2) -- node[pos=.55,sloped,fill=white,inner sep=.6pt] {$a^-_2$} (v4);
\draw[edge] (v2) to[bend left=18] node[pos=.53,sloped,fill=white,inner sep=.6pt] {$a^-_3$} (v4);
\draw[edge] (v1r) to[bend left=18] node[pos=.43,sloped,fill=white,inner sep=.6pt] {$b^+_1$} (v5);
\draw[edge] (v1r) -- node[pos=.46,sloped,fill=white,inner sep=.6pt] {$b^+_2$} (v5);
\draw[edge] (v1r) to[bend right=18] node[pos=.49,sloped,fill=white,inner sep=.6pt] {$b^+_3$} (v5);
\draw[edge] (v1r) to[bend right=18] node[pos=.49,sloped,fill=white,inner sep=.6pt] {$b^-_1$} (v6);
\draw[edge] (v1r) -- node[pos=.46,sloped,fill=white,inner sep=.6pt] {$b^-_2$} (v6);
\draw[edge] (v1r) to[bend left=18] node[pos=.43,sloped,fill=white,inner sep=.6pt] {$b^-_3$} (v6);
\end{tikzpicture}
\end{center}
For $j=1,2,3$, set
\[
 u_j^\pm=p_\pm a_j^\pm,\qquad
 v_j^\pm=q_\pm b_j^\pm,
\]
and form the four column vectors
\[
\mathbf u^\pm=(u_1^\pm,u_2^\pm,u_3^\pm)^{\mathsf T},
\qquad
\mathbf v^\pm=(v_1^\pm,v_2^\pm,v_3^\pm)^{\mathsf T}.
\]
For $\mathbf t=(t_1,t_2,t_3)\in\Bbbk^3$, define the potential $W(\mathbf t)$ on the split quiver by
\[
%\small
W(\mathbf t)=\pi\bigl((\mathbf u^*)^{\mathsf T}\mathbf C_{\mathbf t}\mathbf v^*\bigr)
=\sum_{j=1}^3\pi\bigl(t_j u_j^+v_j^++u_j^+v_j^-+u_j^-v_j^++u_j^-v_j^-\bigr),
\]
where
\[
%\small
\mathbf u^*=\binom{\mathbf u^+}{\mathbf u^-},\quad
\mathbf v^*=\binom{\mathbf v^+}{\mathbf v^-},\quad
\mathbf C_{\mathbf t}=\begin{pmatrix}D_{\mathbf t}&I_3\\ I_3&I_3\end{pmatrix},\quad
D_{\mathbf t}=\operatorname{diag}(t_1,t_2,t_3).
\]

Give $p_\pm,q_\pm$ weight $1$, and assign the remaining
arrow weights by
\[
\begin{array}{c|ccc}
 &j=1&j=2&j=3\\ \hline
 \zeta(a_j^+)=\zeta(a_j^-)&2&2&1\\
 \zeta(b_j^+)=\zeta(b_j^-)&1&1&2.
\end{array}
\]
The paths $u_j^\pm=p_\pm a_j^\pm$ and
$v_j^\pm=q_\pm b_j^\pm$ have weights
\[
\begin{aligned}
 (\zeta(u_1^+),\zeta(u_2^+),\zeta(u_3^+))
 &=(\zeta(u_1^-),\zeta(u_2^-),\zeta(u_3^-))=(3,3,2),\\
 (\zeta(v_1^+),\zeta(v_2^+),\zeta(v_3^+))
 &=(\zeta(v_1^-),\zeta(v_2^-),\zeta(v_3^-))=(2,2,3).
\end{aligned}
\]

\begin{lemma}\label{lem:appendix-trace-formula}
The potential $W=W(\mathbf t)$ is $\zeta$-homogeneous of
weight $5$ for every $\mathbf t=(t_1,t_2,t_3)\in\Bbbk^3$.
If $t_1,t_2,t_3\notin\{0,1\}$, then
$W$ is Jacobi-finite, and
\[
\begin{aligned}
 H_{\zeta,W}(t)
 &=6+10t+24t^2+40t^3+66t^4+94t^5
   +96t^6+72t^7+24t^8,\\
 T_{\zeta,W}(t)
 &=6+\sum_{\substack{1\le j,k\le3\\t_j=t_k}}
 t^{\gamma_j^++\gamma_k^-},\qquad \gamma^+=(2,2,3),\quad \gamma^-=(3,3,2).
\end{aligned}
\]
\end{lemma}
\begin{proof}
Each summand of $W$ has $\zeta$-weight $5$.
Assume $t_j\notin\{0,1\}$ for $1\le j\le3$, and put
$\Lambda=\Lambda(W)$.
We compute the Jacobian algebra by the Gr\"obner basis theory
for quiver path algebras \cite{FFG93}.
Since the cyclic derivatives are homogeneous in ordinary path
length (as well as in $\zeta$-weight), the same computation applies
to the complete path algebra.

For $s\in\{+,-\}$, put $c_{j,+}=t_j$ and $c_{j,-}=1$.
The sixteen cyclic derivatives are
\[
\begin{array}{@{}l@{\quad}l@{\quad}l@{}}
\displaystyle\partial_{a_j^s}W=(c_{j,s}v_j^++v_j^-)p_s,
&
\displaystyle\partial_{b_j^s}W=(c_{j,s}u_j^++u_j^-)q_s,
& s\in\{+,-\},\quad 1\le j\le3,\\[3pt]
\displaystyle\partial_{p_s}W=\sum_{j=1}^3a_j^s(c_{j,s}v_j^++v_j^-),
&
\displaystyle\partial_{q_s}W=\sum_{j=1}^3b_j^s(c_{j,s}u_j^++u_j^-),
& s\in\{+,-\}.
\end{array}
\]
Choose the length-lexicographic order on paths induced by the following order on arrows:
\[
\begin{gathered}
p_+<p_-<q_+<q_-<a_1^+<a_2^+<a_3^+<a_1^-<a_2^-<a_3^-\\
{}<b_1^+<b_3^+<b_2^+<b_1^-<b_2^-<b_3^-.
\end{gathered}
\]
The nonstandard choice $b_3^+<b_2^+$ yields a relatively simple
Gr\"obner--Shirshov basis of the Jacobian ideal, consisting of the
sixteen cyclic derivatives above and the following thirty-eight
relations, grouped by ordinary length:
\[
\begin{aligned}
&\sum_{j=1}^3(t_j-1)a_j^s v_j^+p_{-s},&&
\sum_{j=1}^3(t_j-1)b_j^s u_j^+q_{-s},
\qquad s\in\{+,-\};
\end{aligned}
\]
\[
\begin{aligned}
u_i^+v_j^+p_s, &\qquad s\in\{+,-\},\quad 1\le i,j\le3,
\quad(s,i,j)\ne(-,3,3),\\
v_i^+u_j^+q_s, &\qquad s\in\{+,-\},\quad 1\le i,j\le3,
\quad(s,i,j)\ne(-,2,2).
\end{aligned}
\]
There are $16+4+34=54$ relations altogether. The standard paths
are precisely those avoiding their leading paths. Their number and $\zeta$-weight distribution are
\[
\begin{array}{c|r|l}
\text{length}&\text{number}&\text{$\zeta$-weight polynomial}\\ \hline
0&6&6\\
1&16&10t+6t^2\\
2&36&18t^2+18t^3\\
3&80&22t^3+42t^4+16t^5\\
4&102&24t^4+54t^5+24t^6\\
5&192&24t^5+72t^6+72t^7+24t^8\\
\ge6&0&0
\end{array}
\]
Since these relations are homogeneous in ordinary path length and every
path of length six contains one of their leading paths, there are no
standard paths of length at least six, and
\begin{equation}\label{eq:appendix-simple-nilpotence}
\mm^6\subseteq\mJ(W).
\end{equation}
Thus $W$ is Jacobi-finite. Summing the last column gives the stated
formula for $H_{\zeta,W}(t)$.

We next compute the trace Hilbert series. Let
$\mF_{[4]}\subseteq\mF$ be the subspace spanned by paths of
ordinary length four, and put $\mJ(W)_{[4]}=\mF_{[4]}\cap\mJ(W)$.
The quiver has no cycles of lengths $1$, $2$, $3$, or $5$.
By \eqref{eq:appendix-simple-nilpotence}, there are canonical isomorphisms as graded vector spaces
\[
\pi(\ell+\mF_{[4]})/\pi(\mJ(W)_{[4]})
\cong\pi(\mF)/\pi(\mJ(W))
\cong\operatorname{Tr}(\Lambda),
\]
with grading structures given by $\zeta$-weights.
The space $\pi(\mF_{[4]})$ has basis
\[
\pi(u_j^+v_k^+),\quad \pi(u_j^+v_k^-),\quad
\pi(u_j^-v_k^+),\quad \pi(u_j^-v_k^-),
\qquad 1\le j,k\le3.
\]
Denote these four basis elements by
$x_{jk},y_{jk},z_{jk},w_{jk}$, respectively.
Since $\mJ(W)$ is generated by the cubic cyclic derivatives,
$\mJ(W)_{[4]}$ is spanned by their products with arrows.
Taking cyclic classes, we find that $\pi(\mJ(W)_{[4]})$ is spanned by
the entries of
\[
\begin{pmatrix}
t_k&1&0&0\\
t_j&0&1&0\\
0&1&0&1\\
0&0&1&1
\end{pmatrix}
\begin{pmatrix}x_{jk}\\y_{jk}\\z_{jk}\\w_{jk}\end{pmatrix},
\qquad 1\le j,k\le3.
\]
Here the cyclic derivatives with respect to $p_s,q_s$ contribute only
linear combinations of these entries. The matrix has rank $3$
when $t_j=t_k$ and rank $4$ otherwise.
It follows that
\[
H_{\pi(\mF_{[4]})}(t)=4\sum_{1\le j,k\le3}t^{\gamma_j^-+\gamma_k^+},\qquad
H_{\pi(\mJ(W)_{[4]})}(t)=H_{\pi(\mF_{[4]})}(t)
-\sum_{\substack{1\le j,k\le3\\t_j=t_k}}t^{\gamma_j^-+\gamma_k^+}.
\]
Consequently,
$T_{\zeta,W}(t)=6+H_{\pi(\mF_{[4]})}(t)-H_{\pi(\mJ(W)_{[4]})}(t)$,
as required.
\end{proof}

We next record the right-equivalence criterion for this family.

\begin{lemma}\label{lem:appendix-right-equivalence}
For $\mathbf t=(t_1,t_2,t_3)$ and $\mathbf s=(s_1,s_2,s_3)$ in
$\Bbbk^3$, the potentials $W(\mathbf t)$ and $W(\mathbf s)$ are
right equivalent if and only if $\mathbf t$ is obtained from
$\mathbf s$ by a permutation of coordinates.
\end{lemma}
\begin{proof}
If $\mathbf t$ is obtained from $\mathbf s$ by a permutation of
coordinates, the corresponding simultaneous permutation of the indices
of $a_j^\pm$ and $b_j^\pm$ carries one potential to the other.

Conversely, suppose that a right equivalence $\alpha$ carries
$W(\mathbf t)$ to $W(\mathbf s)$. Since both potentials have ordinary
length four, only the linear part of $\alpha$ contributes to their
length-4 terms. The linear part of $\alpha$ preserves each of the four subspaces of the arrow space $\Bbbk Q_{1}$ spanned by the parallel-arrow families $\{a_j^+\}_{j=1}^3$, $\{a_j^-\}_{j=1}^3$, $\{b_j^+\}_{j=1}^3$, and $\{b_j^-\}_{j=1}^3$, and rescales the arrows $p_\pm,q_\pm$. Hence the induced map of $\alpha$ on $\mm^2/\mm^3$
preserves the four subspaces with respective bases
$\{u_j^+\}_{j=1}^3$,
$\{v_j^+\}_{j=1}^3$,
$\{u_j^-\}_{j=1}^3$, and
$\{v_j^-\}_{j=1}^3$ modulo $\mm^3$.
Using the path vectors $\mathbf u^\pm$ and $\mathbf v^\pm$ defined above,
there are invertible matrices $L_1,L_2,R_1,R_2\in\operatorname{GL}_3(\Bbbk)$
such that
\[
\alpha(\mathbf u^+)\equiv L_1\mathbf u^+,
\quad \alpha(\mathbf u^-)\equiv L_2\mathbf u^-,
\quad \alpha(\mathbf v^+)\equiv R_1\mathbf v^+,
\quad \alpha(\mathbf v^-)\equiv R_2\mathbf v^-
\tag*{$(\operatorname{mod}\,\mm^3)$\qquad}
\]
Comparing the length-4 terms on both sides of
$\alpha(W(\mathbf t))=W(\mathbf s)$ yields
\[
 \begin{pmatrix}L_1^{\mathsf T}&0\\0&L_2^{\mathsf T}\end{pmatrix}
 \mathbf C_{\mathbf t}
 \begin{pmatrix}R_1&0\\0&R_2\end{pmatrix}
 =\mathbf C_{\mathbf s}.
\]
Equating the four blocks gives
\[
 L_1^{\mathsf T}D_{\mathbf t}R_1=D_{\mathbf s},\qquad
 L_1^{\mathsf T}R_2=I_3,\qquad
 L_2^{\mathsf T}R_1=I_3,\qquad
 L_2^{\mathsf T}R_2=I_3.
\]
The last three identities imply $L_1=L_2$ and $R_1=R_2$. Hence
\[
 D_{\mathbf s}=L_1^{\mathsf T}D_{\mathbf t}
 (L_1^{\mathsf T})^{-1},
\]
so $D_{\mathbf s}$ and $D_{\mathbf t}$ are similar. Their eigenvalues,
counted with multiplicity, therefore agree, so $\mathbf t$ is
obtained from $\mathbf s$ by a permutation of coordinates.
\end{proof}

We now specialize to
\[
 \Phi=W(2,2,3),\qquad \Psi=W(2,3,2).
\]
By Lemma~\ref{lem:appendix-trace-formula}, both potentials are
Jacobi-finite and $\zeta$-homogeneous of weight $5$. Since the
parameter triples differ only by a permutation of coordinates,
Lemma~\ref{lem:appendix-right-equivalence} shows that $\Phi$ and
$\Psi$ are right equivalent. The formulas in
Lemma~\ref{lem:appendix-trace-formula} give
$H_{\zeta,\Phi}(t)=H_{\zeta,\Psi}(t)$, whereas
\[
 T_{\zeta,\Phi}(t)=6+5t^5,\qquad
 T_{\zeta,\Psi}(t)=6+t^4+3t^5+t^6.
\]
In particular,
\[
 a_\zeta^{\mathrm{tr}}(\Phi)=5\ne6
       =a_\zeta^{\mathrm{tr}}(\Psi).
\]
This shows that neither the trace Hilbert series nor its degree
is a right-equivalence invariant, even among Jacobi-finite
$\zeta$-homogeneous potentials of the same $\zeta$-weight.

\section*{Acknowledgments}
We thank Professor Osamu Iyama for raising the questions concerning the Jacobi-finiteness and uniqueness of rigid potentials in the fall of 2025. We are also grateful to Dr. Yuhang Cao of the Shanghai Center for Mathematical Sciences and Professor Christof Geiss of UNAM for helpful discussions. Zheng Hua was supported by University Grants Committee of Hong Kong under GRF grant no. 17305123. Guisong Zhou was supported by the National Natural Science Foundation of China (NSFC) under grant no. 12371039.

\section*{Declaration of AI Use}

The authors benefited from discussions with ChatGPT (OpenAI) during the development of this manuscript, including mathematical discussions and assistance with language editing  and LaTeX formatting. The authors take full responsibility for the  mathematical content, correctness, and integrity of the final manuscript,  including the verification of its statements, proofs, examples, and references.

%\bibliographystyle{alpha}
%\bibliographystyle{plain}
%\bibliographystyle{abbrv}

%\bibliography{main}

\Addresses

\end{document}